\documentclass[11pt, oneside]{amsart}

\newcommand{\hide}[1]{}

\usepackage{amssymb}
\usepackage{graphicx}
\usepackage{amsmath}
\usepackage[normalem]{ulem}

\usepackage{mathrsfs}

\usepackage{geometry}
\usepackage{caption}

\usepackage{float}

\usepackage{bm}

\usepackage{wrapfig}

\usepackage{xcolor}

\usepackage{enumitem}

\usepackage{tikz}
\usetikzlibrary{arrows, arrows.meta}

\def\textcolor#1{}

\newcommand{\R}{\mathbb{R}}

\newcommand{\Hh}{\mathbb{H}}
\renewcommand{\S}{\mathbb{S}}

\newcommand{\sA}{\mathscr{A}}
\newcommand{\sM}{\mathscr{V}}

\newcommand{\Ss}{\mathcal{S}}
\newcommand{\Bb}{\mathcal{B}}

\newcommand{\rr}{\mathfrak{r}}

\renewcommand{\rho}{\varrho}
\renewcommand{\phi}{\varphi}
\newcommand{\eps}{\varepsilon}
\renewcommand{\theta}{\vartheta}

\DeclareMathOperator{\Rr}{{\mathcal R}}

\DeclareMathOperator{\sn}{sn}
\DeclareMathOperator{\cn}{cn}

\DeclareMathOperator{\dist}{dist}

\renewcommand{\ge}{\geqslant}
\renewcommand{\le}{\leqslant}

\renewcommand{\geq}{\geqslant}

\usepackage[colorlinks=true,urlcolor=blue]{hyperref}

\theoremstyle{theorem}
\newtheorem{theorem}{Theorem}[section]
\newtheorem{lemma}[theorem]{Lemma}

\newtheorem*{BorisenkoConj}{Borisenko's Conjecture}
\newtheorem*{main}{Main Theorem}

\newtheoremstyle{Intro}
{}
{}
{\itshape}
{}
{\scshape}
{.}
{.5em}
{}

\theoremstyle{Intro}

\theoremstyle{definition}
\newtheorem{definition}[theorem]{Definition}

\theoremstyle{remark}
\newtheorem*{remark}{\textsc{Remark}}

\newcounter{reminder}

\newtheoremstyle{claim}
  {}
  {}
  {\itshape}
  {0pt}
  {\scshape}
  {.}
  { }
  {\thmname{#1}\thmnumber{ #2}\thmnote{ (#3)}}

\theoremstyle{claim}

\numberwithin{equation}{section}

\definecolor{barriergreen}{RGB}{2,43,58}
\definecolor{bodycolor}{RGB}{31,122,140}
\definecolor{neighbourblue}{RGB}{202,160,82}
\definecolor{labelgreen}{RGB}{2,43,58}

\definecolor{barriergreen}{RGB}{2,43,58}      
\definecolor{bodycolor}{RGB}{224,122,95}      
\definecolor{neighbourblue}{RGB}{139,112,177} 
\definecolor{labelgreen}{RGB}{2,43,58}        

\definecolor{neighbourblue}{RGB}{112,153,121} 

\definecolor{neighbourblue}{RGB}{76,175,80} 

\definecolor{ttzzqq}{rgb}
{0.12156862745098039, 0.47843137254901963, 0.5490196078431373}

\title[A complete solution to the reverse isoperimetric problem]{A complete solution to the reverse isoperimetric problem for $\lambda$-convex bodies}

\author[K. Drach]{Kostiantyn Drach}

\author[K. Tatarko]{Kateryna Tatarko}

	\thanks{
    \indent The first author is partially supported by Agencia Estatal de Investigaci\'on grants PID2023-147252NB-I00, CNS2025-166633, and by the Severo Ochoa and Mar\'ia de Maeztu Program for Centers and Units of Excellence in R\&D (CEX2020-001084-M); the second author is partially supported by NSERC Discovery Grant number 2022-02961. }


\date{}

\address{Departament de Matem{\`a}tiques i Inform{\`a}tica, Universitat de Barcelona, Gran Via de les Corts Catalanes 585, 08007 Barcelona, Spain \newline\hspace*{1em}
Centre de Recerca Matem{\`a}tica, Edifici C, Carrer de l'Albareda, 08193 Bellaterra, Spain}
\email{kostiantyn.drach@ub.edu}
\address{Department of Pure Mathematics, University of Waterloo, Waterloo, ON, N2L 3G1, Canada}
\email{ktatarko@uwaterloo.ca}

\begin{document}

\maketitle

\begin{center}
    {\it  Dedicated to Alexander Borisenko on the occasion of his 80th birthday.}
\end{center}

\vspace{0.3 cm}

\begin{abstract}

For a given $\lambda > 0$, a convex body in $\mathbb{R}^n$ is called $\lambda$-convex if it is the intersection of balls of radius $1/\lambda$. We show that among all $\lambda$-convex bodies in $\mathbb R^n$ of a given surface area, volume is minimized by the $\lambda$-convex lens, that is, the intersection of two balls of radius $1/\lambda$. This confirms a conjecture of Borisenko in all dimensions $n \geq 2$.  Moreover, we extend this reverse isoperimetric inequality to all model spaces of constant curvature in arbitrary dimension and arbitrary $\lambda > 0$.
\end{abstract}

\section{introduction}

A convex body in $\R^n$ is a compact convex set with non-empty interior. Given $\lambda>0$, a convex body $K\subset\R^n$ is called \emph{$\lambda$-convex} if, at every point of its boundary, it is locally contained in a ball of radius $1/\lambda$ whose boundary passes through that point. More precisely, for each $p \in \partial K$ there exist a neighborhood $U_p$ of $p$ and a ball $\mathcal B_{\lambda}$ of radius $1/\lambda$ such that $p \in \partial \mathcal B_{\lambda}$ and $U_p \cap K\subset \mathcal B_{\lambda}$ (see Figure~\ref{Fig:Def}). When $\partial K$ is smooth, this condition is equivalent to requiring that its principal curvatures, computed with respect to the inward-pointing unit normal, satisfy $k_i(p)\geq\lambda$ for every $p\in\partial K$ and $i=1,\ldots,n-1$. 

By the classical Blaschke Rolling Theorem \cite{Bla56} (see also \cite{DrBla} and references therein for a historical account of this result), every $\lambda$-convex body is an intersection of balls of radius $1/\lambda$. 

One of the key examples of a $\lambda$-convex body is a \emph{$\lambda$-convex lens} defined as the intersection of two balls of radius $1/\lambda$ (see Figure~\ref{Fig:Lens}). 

The geometry of $\lambda$-convex bodies has been studied extensively in convex and discrete geometry over the past two decades. We refer the reader to \cite{BLN} and more recent \cite[Section 6]{AA} for surveys on $\lambda$-convexity and related questions (see also \cite{BLNP, AAF}).

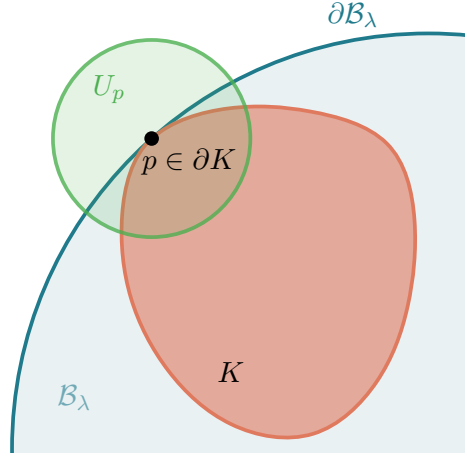
\begin{figure}
\begin{tikzpicture}[x=0.55cm,y=0.55cm,  line width=1.4pt,line cap=round,line join=round,
  every node/.style={inner sep=0pt}]
\coordinate (p) at (3.354,7.558);
\begin{scope}
  \clip (-0.15,0) rectangle (10.863,11.0);
  \path[draw=ttzzqq,draw opacity=1.,
        fill=ttzzqq,fill opacity=.1]
    (10,0.086) circle[radius=10];
\end{scope}
\path[draw=bodycolor,fill=bodycolor,fill opacity=0.60]
  (p)
  .. controls (3.836,7.986) and (4.534,8.200) .. (5.212,8.284)
  .. controls (5.891,8.368) and (6.490,8.330) .. (7.155,8.210)
  .. controls (7.820,8.091) and (8.552,7.890) .. (9.002,7.440)
  .. controls (9.452,6.991) and (9.619,6.294) .. (9.669,5.590)
  .. controls (9.718,4.887) and (9.649,4.178) .. (9.504,3.523)
  .. controls (9.359,2.869) and (9.138,2.269) .. (8.763,1.714)
  .. controls (8.389,1.159) and (7.861,0.649) .. (7.234,0.457)
  .. controls (6.607,0.265) and (5.880,0.391) .. (5.268,0.686)
  .. controls (4.656,0.981) and (4.159,1.445) .. (3.756,1.973)
  .. controls (3.354,2.501) and (3.045,3.093) .. (2.856,3.725)
  .. controls (2.667,4.358) and (2.599,5.030) .. (2.645,5.714)
  .. controls (2.692,6.397) and (2.873,7.130) .. (p)
  -- cycle;
\path[draw=neighbourblue,draw opacity=0.80, fill=neighbourblue,fill opacity=0.15]
  (p) circle[radius=2.371];
\fill (p) circle[radius=0.175];
\node at (5.26,1.94) {$K$};
\node[anchor=north west] at (3.08,7.36) {$p\in\partial K$};
\node[text=neighbourblue] at (2.33,8.75) {$U_p$};
\node[text=ttzzqq] at (8.19,10.53) {$\partial \mathcal B_\lambda$};
\node[text=ttzzqq, text opacity=0.6] at (1.5,1.33) {$\mathcal B_\lambda$};
\end{tikzpicture}
    \caption{The definition of $\lambda$-convexity.}
    \label{Fig:Def}
\end{figure}

This paper focuses on a \emph{reverse} isoperimetric problem under curvature constraints. The classical isoperimetric inequality implies that among all $\lambda$-convex bodies of a given
surface area, the Euclidean ball has the largest volume. The reverse problem of determining the smallest possible volume under the same surface area constraint is one of the central open problems in the field. In particular, Borisenko conjectured  (see \cite{CDT, DrPhD}) that a $\lambda$-convex lens is the unique minimizer. For a more detailed discussion, we refer the reader to the introductions in \cite{DT, DTProc}.

\begin{BorisenkoConj}
Let $n \ge 2$ and $\lambda > 0$.  Let $K \subset \mathbb{R}^n$ be a $\lambda$-convex body and let $L \subset \R^n$ be a $\lambda$-convex lens. If $|\partial K| = |\partial L|$, then $|K| \ge |L|$, with equality if and only if $K$ is a $\lambda$-convex lens.
\end{BorisenkoConj}
Here $|A|$ denotes the $n$-dimensional volume of a measurable set  $A\subset \R^n$ and $|\partial A|$ denotes the $(n-1)$-dimensional volume of its boundary.

For $n=2$, Borisenko and Drach resolved the conjecture in \cite{BorDr14} (see also the alternative proof in \cite{FKV}), and for $ n=3$, Drach and Tatarko confirmed it in \cite{DT}. The corresponding problem was also fully resolved in all two-dimensional model spaces of constant curvature and in certain natural generalizations of these spaces in \cite{BorDr15_1, DrLob, BorNA}. Recently, it was confirmed in every three-dimensional space form in \cite{DTS}. We also note that in \cite{Julian}, it was shown that the solution to Borisenko's conjecture in Euclidean, spherical, and hyperbolic spaces cannot have a smooth boundary.

In this paper, we prove Borisenko's conjecture in every dimension $n\ge 2$. Moreover, we establish the corresponding reverse isoperimetric inequality in all model spaces of constant curvature.

\begin{main}\label{Thm:Main1}
Borisenko's Conjecture is true in $\mathbb{R}^n$, as well as in spherical space $\S^{n}$ and hyperbolic space $\mathbb{H}^n$, for any $n \ge 2$ and $\lambda >0$.
\end{main}
The proof employs the reverse inradius inequality from \cite{DT} with the spherical comparison used in \cite{BezdekConj}.

The paper is organized as follows. In Section~\ref{sec:prelim}, we collect the necessary background on $\lambda$-convexity. In Section~\ref{sec:euclidean-proof}, we prove the main theorem in $\mathbb{R}^n$, where the main ideas of the proof can be presented more transparently. Then, in Section~\ref{sec:model-proof}, we extend the argument to the spherical and hyperbolic spaces $\S^{n}$ and $\mathbb{H}^n$. In the hyperbolic case $\mathbb H^n$, which we assume to be of curvature $-1$, extra care will be given to the \emph{equidistant case} of $\lambda \in (0,1)$. This is due to the fact that, contrary to the case when $\lambda \ge 1$, an intersection of two convex regions bounded by hypersurfaces of constant normal curvature $\lambda$ may not be compact.

\section{Preliminaries}
\label{sec:prelim}

We denote by $M^n(c)$ the complete simply connected $n$-dimensional Riemannian manifold of constant curvature $c$. We call $M^n(c)$ a \emph{space form} (of curvature $c$). Thus, $M^n(c)$ is the standard Euclidean space $\R^n$ when $c = 0$, spherical space $\mathbb S^n(c)$ if $c > 0$ and hyperbolic space $\Hh^n(c)$ if $c < 0$.

For a given $\lambda > 0$, a \emph{$\lambda$-sphere} is a complete totally umbilical hypersurface $\Ss_\lambda \subset M^n(c)$ of constant normal curvature equal to $\lambda$. A \emph{$\lambda$-ball} $\Bb_\lambda$ is the convex and topologically closed region bounded by $\Ss_\lambda$.

When $c \ge 0$, $\lambda$-spheres are geodesic spheres. If $c<0$, $\Ss_\lambda$ is a geodesic sphere if $\lambda > \sqrt{-c}$, a horosphere if $\lambda = \sqrt{-c}$, and an equidistant hypersurface (or equidistant for short) if $\lambda \in (0,\sqrt{-c})$. Here, an equidistant hypersurface is one of the two connected components of the set of points at a given distance from a totally geodesic hyperplane.

\subsection{$\lambda$-convex bodies}	
\label{SSec:LConv}

Fix some $\lambda>0$. 

\begin{definition}[$\lambda$-convex body]
\label{Def:LambdaConvex}
A convex body $K \subset M^n(c)$ is \emph{$\lambda$-convex} if for each $p \in \partial K$ there exists a neighborhood $U_p \subset M^n(c)$ and a $\lambda$-sphere $\Ss_\lambda$ that passes through $p$ in such a way that
\[
U_p \cap \partial K \subset \Bb_\lambda,
\]
where $\Bb_\lambda$ is the corresponding $\lambda$-ball (see Figure~\ref{Fig:Def}).
\end{definition}

In this case, we call $\Ss_\lambda$ a \emph{supporting $\lambda$-sphere} passing through $p \in \partial K$.

The following classical theorem of Blaschke turns the local condition in Definition~\ref{Def:LambdaConvex} into a global one; see \cite{Bla56} for the original proof in the Euclidean space and \cite{DrBla} for an overview of further developments and generalizations.

\begin{theorem}[Blaschke's Rolling Theorem]
\label{Thm:Bla}
Let $K \subset M^n(c)$ be a $\lambda$-convex body. Then for every $p \in \partial K$ there exists a supporting $\lambda$-sphere $\Ss_{\lambda}$ bounding the $\lambda$-ball $\Bb_\lambda$ such that $p\in\Ss_\lambda$ and $K\subseteq  \Bb_{\lambda}$. 
\end{theorem}

\begin{definition}[$\lambda$-convex polytope]
	\label{Def:LambdaConvexPol}
A \emph{$\lambda$-convex polytope} in $M^n(c)$ is a $\lambda$-convex body $K$ given as the intersection of finitely many convex regions $\Bb_\lambda$. The boundary $\partial K$ is piecewise smooth, and the closure of each maximal connected smooth part is called a \emph{facet} of $K$. A facet is of dimension $n-1$ and is equal to the intersection $K \cap \partial \Bb_\lambda$ for some $\Bb_\lambda$.
\end{definition}

\begin{definition}[$\lambda$-convex lens]	\label{Def:LambdaLens}
A \emph{$\lambda$-convex lens} is a $\lambda$-convex body given by the intersection of two sets of the form $\Bb_\lambda$ (see Figure~\ref{Fig:Lens}).
\end{definition}

\begin{figure}[h]
\definecolor{uququq}{rgb}
{0.00784313725490196, 0.16862745098039217, 0.22745098039215686}
\definecolor{ttzzqq}{rgb}
{0.12156862745098039, 0.47843137254901963, 0.5490196078431373}
\begin{tikzpicture}[line cap=round,line join=round,>=triangle 45,x=1.0cm,y=1.0cm]
\clip(-4.,-2.1) rectangle (4.,2.1);
\fill[line width=1.pt,color=ttzzqq,fill=ttzzqq,fill opacity=0.10000000149011612] (-3.842986564339558,0.046330343966855855) -- (3.842986564339557,0.04633034396685584) -- (3.842986564339557,-0.04633034396685584) -- (-3.842986564339558,-0.046330343966855855) -- cycle;
\draw [shift={(0.,-3.0079099572141677)},line width=1.pt,color=ttzzqq,fill=ttzzqq,fill opacity=0.10000000149011612]  plot[domain=0.6715357772283473:2.470056876361446,variable=\t]({1.*4.908862347943025*cos(\t r)+0.*4.908862347943025*sin(\t r)},{0.*4.908862347943025*cos(\t r)+1.*4.908862347943025*sin(\t r)});
\draw [shift={(0.,0.)},line width=1.pt,color=ttzzqq]  plot[domain=-3.239166666629998:0.09757401304020531,variable=\t]({1.*3.861353355677969*cos(\t r)+0.*0.4755768507003521*sin(\t r)},{0.*3.861353355677969*cos(\t r)+1.*0.4755768507003521*sin(\t r)});
\draw [shift={(0.,0.)},line width=0.7pt,dash pattern=on 1.5pt off 2.5pt,color=ttzzqq]  plot[domain=0.09757401304020531:3.0440186405495884,variable=\t]({1.*3.861353355677969*cos(\t r)+0.*0.4755768507003521*sin(\t r)},{0.*3.861353355677969*cos(\t r)+1.*0.4755768507003521*sin(\t r)});
\draw [shift={(0.,3.0079099572141677)},line width=1.pt,color=ttzzqq,fill=ttzzqq,fill opacity=0.10000000149011612]  plot[domain=0.6715357772283473:2.470056876361446,variable=\t]({1.*4.908862347943025*cos(\t r)+0.*4.908862347943025*sin(\t r)},{0.*4.908862347943025*cos(\t r)+-1.*4.908862347943025*sin(\t r)});
\begin{scriptsize}
\draw [fill=uququq, color=uququq] (0.,0.) circle (1.0pt);
\end{scriptsize}
\end{tikzpicture}
\caption{The $\lambda$-convex lens.}
\label{Fig:Lens}
\end{figure}
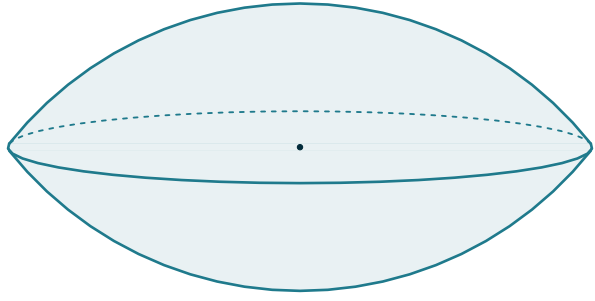


Note that in our definition, a $\lambda$-convex lens can be a ball (if the two generating sets $\Bb_\lambda$ are balls that coincide). Otherwise, a $\lambda$-convex lens is a $\lambda$-convex polytope that has exactly two facets. Each lens is reflection-symmetric with respect to the totally geodesic hyperplane passing through the intersection of the boundaries of the two facets. 

Let $K \subset M^n(c)$ be a convex body. A geodesic ball of the largest possible radius contained in $K$ is called an \emph{inscribed ball for $K$}. The \emph{inradius of $K$}, denoted by $r(K)$, is the radius of an inscribed ball for $K$. Note that $r(K)$ is well-defined even though a convex body can have several inscribed balls.

We will need the following reverse-inradius inequality, proved in \cite{DT}.

\begin{theorem}[Reverse inradius inequality]
    \label{thm:reverse-inradius}
    Let $n \ge 2$. Let $K \subset M^n(c)$ be a $\lambda$-convex body, and let $L \subset M^n(c)$ be a $\lambda$-convex lens. If $|\partial K| = |\partial L|$, then $r(K) \ge r(L)$. Equivalently, if $r(K) = r(L)$, then $|\partial K| \le |\partial L|$.

    Moreover, equality (in each statement) holds if and only if $K$ is a $\lambda$-convex lens. 
\end{theorem}

For $2$-dimensional $\lambda$-convex bodies in $\mathbb{R}^2$, Milka \cite{MilInradius} proved that the $\lambda$-convex lens has the smallest inradius for a given length of the boundary. Much later, this result was generalized in \cite{Dr4} to arbitrary $\lambda$-convex domains in $2$-dimensional Alexandrov spaces of curvature bounded below, thus closing this question in dimension $2$. Bezdek \cite{BezdekConj} showed that among all $\lambda$-convex bodies in $\R^n$, $n\ge 2$, the $\lambda$-convex lens has the smallest inradius for a given \emph{volume}.

\subsection{A spherical lemma.}
\label{subsec:spherical-lemma}

Let $T\subset\mathbb S^{n-1}$ be a finite subset of points in the unit sphere that is not contained in any open hemisphere.

Let $\omega\in\mathbb S^{n-1}$. Define the spherical distance from $\omega$ to the set $T$ as
\begin{equation}\label{spher_dist}
    d_T(\omega) = \textup{dist}_{\mathbb S^{n-1}} (\omega, T) = \min\limits_{x \in T} \textup{dist}_{\mathbb S^{n-1}}(\omega, x)
\end{equation}
where $\text{dist}_{\mathbb S^{n-1}}$ denotes the usual spherical distance.
Fix $e \in \mathbb S^{n-1}$, and denote the distance to the set $\{e, -e\}$ of two antipodal points of $ \S^{n-1}$ by
$$
d_{\pm e}(\omega) = \textup{dist}_{\mathbb S^{n-1}}(\omega, \{e, -e\}).
$$
Since $T$ is not contained in any open hemisphere, then $d_T(\omega) \le \pi/2$ for every $\omega\in\mathbb{S}^{n-1}$. Also, $d_{\pm e}(\omega) \le \pi/2$.

For $\varepsilon \in [0, \pi/2]$, let 
$$
T_{\eps} = \{\omega \in \S^{n-1}: \, d_T(\omega) \le \eps \}
$$
be the spherical $\eps$-neighbourhood of $T$. We write $\{\pm e\}_\varepsilon$ for the corresponding neighbourhood of the pair of antipodal points.  Then we state the following result, which can be found in various forms in \cite{AK,BSch,BezdekConj}. We present its proof here for completeness.

\begin{lemma}
\label{lem:spherical-comparison}
Let $h:[0,\pi/2]\to\mathbb R$ be a non-decreasing function and let $T = \{x_1, \dots, x_m \}\subset\mathbb S^{n-1}$ be a finite subset of points that is not contained in any open hemisphere. Then
$$
\int_{\mathbb S^{n-1}}h\left(d_T(\omega)\right)\,d\omega
\le
\int_{\mathbb S^{n-1}}h\left(d_{\pm e}(\omega)\right)\,d\omega,
$$
where $e\in\mathbb{S}^{n-1}$ is fixed and $d\omega$ denotes spherical surface measure.

In particular, for every $\eps\in[0,\pi/2]$,
$$
|T_{\eps}| \ge |\{\pm e\}_\varepsilon|.
$$
\end{lemma}
\begin{proof}
Consider the Voronoi tiling of $\S^{n-1}$ generated by $T$ and let 
$$
V_i = \{\omega\in\mathbb S^{n-1}:\,
\textup{dist}_{\mathbb S^{n-1}}(\omega,x_i)
=d_T(\omega)\}
$$
be the Voronoi cell of $x_i \in T$, $1 \le i\le m $. We note that each $V_i$ is a star-shaped set with respect to $x_i.$ Since $d_T(\omega) \le \pi/2$, then $V_i$ lies in the closed hemisphere of $\S^{n-1}$ centered at $x_i$. 

Choose spherical polar coordinates $(t, \theta)$ in $\S^{n-1}$ centered at $x_i$ such that a point $\omega \in \S^{n-1}$ can be written as
\[
\omega = \cos t\, x_i + \sin t \, \theta,
\]
where $\theta \in \S^{n-2}$ (more precisely, $\theta \in x_i^\perp\cap\mathbb S^{n-1}$) and $t \in [0, {\pi}/2]$ is the angle between $\omega$ and direction $x_i$.  Then the surface area measure on $\S^{n-1}$ is $d\omega = \left(\sin (t)\right)^{n-2} dt \, d\theta$.

Denote the spherical radial function of $V_i$ in the direction $\theta$ by 
$$
\rho_i (\theta) := \max\left\{ t \in [0,{\pi}/2]: \  \cos(t) x_i + \sin (t) \, \theta \in V_i  \right\}.
$$
Note that $\rho_i (\theta) \in [0, \pi/2]$.

Using polar coordinates, we get
\begin{align} \label{Eq::split}
\int_{\mathbb S^{n-1}}h\left(d_T(\omega)\right)\,d\omega &= \sum\limits_{i=1}^m  \int_{V_i} h\left(d_T(\omega)\right)\,d\omega  = \sum\limits_{i=1}^m  \int_{\S^{n-2}} \int\limits_0^{\rho_i(\theta)} h\left(t\right) (\sin t)^{n-2}\, dt d\theta.
\end{align}
Let $\mathcal{F} = \displaystyle\int_0^{\pi/2}
h(t)(\sin t)^{n-2} \,dt \bigg/\displaystyle\int_0^{\pi/2}
(\sin t)^{n-2} \,dt.$ We claim that 
\begin{equation}\label{Eq::ratio}
     \int\limits_0^{r} h\left(t\right) (\sin t)^{n-2}\, dt   \le \mathcal{F} \int\limits_0^{r} (\sin t)^{n-2}\, dt
\end{equation}
for any $0 < r \le {\pi}/2$ (compare with \cite[Lemma 4.4]{DT}). When $r = \pi/2$, the conclusion follows immediately. Assume that $0 < r < {\pi}/2.$ Indeed, since $h$ is non-decreasing, then $h(x) \le h(r) \le h(y)$ for any $x \in [0, r]$ and $y \in [r, \pi/2].$ Thus, the weighted average of $h$ on the interval $[0, r]$ is no greater than the weighted average of $h$ on $[r, \pi/2].$ Moreover, the weighted average $\mathcal{F}$ of $h$ on $[0, \pi/2]$  lies between these two averages which proves \eqref{Eq::ratio}.

Using \eqref{Eq::ratio}, we now can continue the estimate in \eqref{Eq::split}:
\begin{align*} 
& \sum\limits_{i=1}^m  \int_{\S^{n-2}} \int\limits_0^{\rho_i(\theta)} h\left(t\right) (\sin t)^{n-2}\, dt d\theta \le \mathcal{F} \sum\limits_{i=1}^m  \int_{\S^{n-2}} \int\limits_0^{\rho_i(\theta)} (\sin t)^{n-2}\, dt d\theta \\
 &= \mathcal{F}  |\S^{n-1}| =  \int_{\mathbb S^{n-1}}h(d_{\pm e}(\omega))\,d\omega
\end{align*}
which proves the first assertion.

Applying the first inequality with $h(t) = \mathbf 1_{(\eps,\pi/2]}(t)$, we get 
$$
|\{\omega \in \S^{n-1}: \ d_T(\omega) > \eps \}| \le |\{\omega \in \S^{n-1}: \ d_{\pm e}(\omega) > \eps \}|,
$$
that is, $|\mathbb S^{n-1}\setminus T_\varepsilon| \le |\mathbb S^{n-1}\setminus \{\pm e\}_\varepsilon|$. The latter implies that $|T_{\eps}| \ge |\{\pm e\}_\varepsilon|$.

\end{proof}

\section{Proof of the Main Theorem in Euclidean space}
\label{sec:euclidean-proof}

Without loss of generality, we can assume that $\lambda = 1$ and $K$ is a $1$-convex body. 

Let $B = B(o,r)$ be the inscribed ball for $K$ of radius $r := r(K)$ centered at the origin $o \in \R^n$. Note that since $K \subset \R^n$ is strictly convex, this ball is unique. Let $L$ be a $1$-convex lens with $|\partial L|=|\partial K|$. By Theorem~\ref{thm:reverse-inradius}, we have that $r \ge r_0 := r(L)$. If $r=1$, then $K$ is a ball of radius $1$, and the statement of the theorem follows. We therefore can assume that 
\begin{equation}
\label{Eq:InradIneq}
r_0 \le r< 1.
\end{equation}

We first prove the theorem when $K$ is a $1$-convex polytope. Then, the general case will follow by approximation and rigidity in the reverse inradius inequality (Theorem~\ref{thm:reverse-inradius}) as explained in Subsection~\ref{SSec:General}. We consider two cases: 
\begin{enumerate}[label=(\Alph*)]
    \item 
    \label{It:Case1}
    All facets of the polytope $K$ touch its inscribed ball $B$;
    \item
    \label{It:Case2}
    At least one facet of $K$ does not touch $B$.
\end{enumerate}

It turns out that the argument in case \ref{It:Case1} is the most transparent, so we will carry out the proof for this case first. Afterwards, we will indicate how to modify the proof to address case \ref{It:Case2}.

\subsection{Case \ref{It:Case1}: All facets of $K$ touch $B$}
\label{SSec:Touch}

Let $S_1$ be a unit sphere touching $B$. We can parametrize $S_1$ by its radial function, denoted by $\ell_r(\omega)$, $\omega \in \S^{n-1}$ where we emphasize the dependence on the radius $r$ of $B$. Assume that the touching point $S_1 \cap \partial B$ corresponds to $\omega_0 \in \S^{n-1}$. Let $\theta = \dist_{\S^{n-1}} (\omega_0, \omega)$ be the spherical distance between $\omega_0$ and $\omega$ (or, equivalently, the angle between the rays in $\R^n$ in the directions $\omega_0$ and $\omega$). By rotational symmetry, $\ell_r$ can be written explicitly as a function of $\theta$
\[
\ell_r(\theta) = \cos \phi_r(\theta)-(1-r)\cos \theta,
\]
where $\phi_r(\theta)$ is the angle between the ray in $\R^n$ in the direction $\omega$ and the outer normal to $S_1$ at the corresponding point of intersection. We call $\phi_r$ the \emph{radial angle}. This angle was extensively studied earlier in the context of $1$-convex bodies (see, e.g., \cite{
BorDr13, BDr15, BGR, DrPhD, DrBla} for various comparison-type results and their applications).

A direct computation gives
\[
\sin \phi_r(\theta)=(1-r)\sin \theta, \qquad \cos \phi_r(\theta) = \sqrt{1-(1-r)^2\sin^2 \theta}.
\]
Note that $\ell_r$ is strictly increasing on $[0,\pi/2]$. Finally, it is well-known (see, e.g., \cite{BorArea}) that the surface area element of $S_1$ and the volume element of the ball bounded by $S_1$ are given by, respectively,
\begin{equation}
\label{Eq:AreaVol}
A_r(\theta):=\frac{\ell_r(\theta)^{n-1}}{\cos \phi_r(\theta)}, \qquad V_r(\theta):=\frac{\ell_r(\theta)^n}{n}.
\end{equation}

Let $Q \subset \S^{n-1}$ be the set of radial projections of the points where the facets of $K$ touch $B$ onto the unit sphere (recall that both spheres $\partial B$ and $\S^{n-1}$ are centered at the origin). Also, $Q$ does not lie in any open hemisphere.
The radial projections of the facets of $K$ onto $\S^{n-1}$ are the Voronoi cells generated by $Q$ (see, e.g., \cite[Section 3]{DTProc}). Therefore, using notation of Section~\ref{subsec:spherical-lemma}, the radial function $\rho_K$ of $K$ satisfies 
\begin{equation}
    \label{Eq:Bound}
    r\leqslant\rho_K(\omega) = \ell_r(d_Q(\omega)), \qquad  \omega \in \S^{n-1}.
\end{equation}
Then by Lemma~\ref{lem:spherical-comparison},
\begin{equation}
    \label{Eq:Comp0}
    |\{\omega \in \S^{n-1} : d_Q(\omega) > s\}| \le |\{\omega \in \S^{n-1} : d_{\pm e}(\omega) > s\}| \qquad \textup{for } s \in [0,\pi/2].
\end{equation}

It is well-known (see, e.g., \cite{BorArea}) that the volume and the surface area of $K$ can be written in terms of the radial function as follows: 
\begin{equation}
    \label{Eq:AVSupport}
    |\partial K| = \int_{\S^{n-1}} \frac{\rho_K(\omega)^{n-1}}{\cos \phi_K(\omega)}d\omega, \qquad |K| = \int_{\S^{n-1}} \frac{\rho_K(\omega)^n}{n}d\omega,
\end{equation}
where $d\omega$ denotes the standard surface measure on $\S^{n-1}$ and $\phi_K(\omega)$ is the radial angle function of $K$ with respect to the origin. Recall that $\phi_K(\omega)$ is defined as the angle between the ray $w$ from the origin in the direction $\omega \in \S^{n-1}$ and the outer normal to the supporting hyperplane for $K$ at the point $\partial K \cap w = \rho_K(\omega) \cdot \omega$. Note that the set where this is not well-defined has zero surface measure. Using \eqref{Eq:AreaVol} and \eqref{Eq:Bound}, we get
\[
|\partial K| = \int_{\S^{n-1}} A_r(d_Q(\omega)) d\omega, \qquad |K| = \int_{\S^{n-1}}V_r(d_Q(\omega)) d\omega.
\]

We now introduce the main idea of the proof: a \emph {cutoff} of the distance function $d_{\pm e}$ defined for $T\in[0,\pi/2]$ as follows 
\[
d_{\pm e,  T}(\omega):=\min\{d_{\pm e}(\omega),T\}.
\]
Thus, $d_{\pm e,T}$ equals $T$ outside the geodesic balls of radius $T$ in $\S^{n-1}$ centered at $\pm e$, and equals the distance to the nearest of the points $\pm e$ inside these balls. For $s \in [r_0, r]$ and $T \in [0, \pi/2]$, we define the function
\begin{equation}
\label{Eq:sA}
\mathscr{A} \colon (s,T) \longmapsto\int_{\mathbb S^{n-1}}A_s\left(d_{\pm e,T}(\omega)\right)\,d\omega.
\end{equation}
This function is smooth and strictly monotone in both variables on its domain of definition. Moreover, $\mathscr{A}$ is independent of $K$. 
Also,  
\[
\sA(s,0) = |\partial B(o,s)|, \qquad \sA(s,\pi/2) = |\partial L_s|,
\]
where $L_s$ is a 1-convex lens with the inradius $s$. Since $s \le r$ and $B(0,r) \subseteq K$ then the monotonicity of surface area gives $\sA(s,0) \le |\partial K|.$ On the other hand, $\sA(s,\pi/2) = |\partial L_s| \ge |\partial L_{r_0}| = |\partial K|.$ Therefore, $\sA(s,0) \le  |\partial K| \le \sA(s,\pi/2)$ for every $s \in [r_0, r]$. Thus, by the intermediate value theorem and the strict monotonicity of $\sA$, there exists a unique $T(s) \in [0, \pi/2]$  such that
\begin{equation}
    \label{Eq:Area}
    \sA(s,T(s)) = |\partial K|.
\end{equation}

For such $T(s)$, we define
\begin{equation}
\label{Eq:Vol}
\sM(s) := \int_{\mathbb S^{n-1}}V_s\left(d_{\pm e,T(s)}(\omega)\right)\,d\omega.
\end{equation}
(see Figure~\ref{Fig:Truncation}). Note that $T(r_0) = \pi/2$ and 
\begin{equation}
\label{Eq:Final}
\sM(r_0) = |L|.
\end{equation}

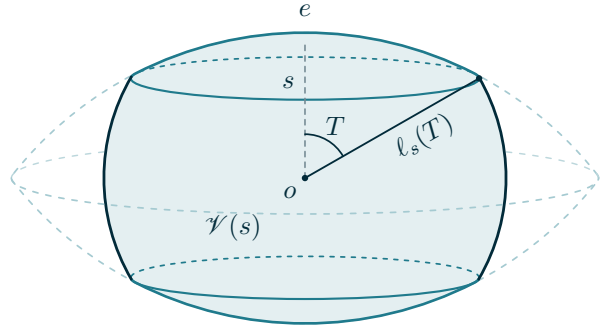
\begin{figure}[ht]
\centering
\definecolor{uququq}{rgb}{0.0078431373,0.1686274510,0.2274509804}
\definecolor{ttzzqq}{rgb}{0.1215686275,0.4784313725,0.5490196078}
\begin{tikzpicture}[
  line cap=round,line join=round,
  x=1cm,y=1cm,font=\small,
  every node/.style={text=uququq}
]
  \pgfmathsetmacro{\lensradius}{4.9088623479}
  \pgfmathsetmacro{\lensoffset}{3.0079099572}
  \pgfmathsetmacro{\cutangle}{60} 
  \pgfmathsetmacro{\lensheight}{\lensradius-\lensoffset}
  \pgfmathsetmacro{\rimradius}{sqrt(\lensradius^2-\lensoffset^2)}
  \pgfmathsetmacro{\cutradius}{
    sqrt(\lensradius^2-\lensoffset^2*sin(\cutangle)^2)
    -\lensoffset*cos(\cutangle)}
  \pgfmathsetmacro{\cutx}{\cutradius*sin(\cutangle)}
  \pgfmathsetmacro{\cuty}{\cutradius*cos(\cutangle)}
  \pgfmathsetmacro{\viewdepth}{.4756/\rimradius}
  \pgfmathsetmacro{\viewheight}{sqrt(1-\viewdepth^2)}
  \pgfmathsetmacro{\ringheight}{\cuty*\viewheight}
  \pgfmathsetmacro{\ringdepth}{\viewdepth*\cutx}
  \pgfmathsetmacro{\capx}{sqrt(\lensradius^2-
    ((\lensoffset+\cuty)/\viewheight)^2)}
  \pgfmathsetmacro{\capy}{(\cuty+\lensoffset*\viewdepth^2)/\viewheight}
  \pgfmathsetmacro{\bandx}{sqrt(\cutradius^2-(\cuty/\viewheight)^2)}
  \pgfmathsetmacro{\bandy}{\cuty/\viewheight}
  \pgfmathsetmacro{\capangle}{acos(\capx/\lensradius)}
  \pgfmathsetmacro{\bandangle}{asin(\bandy/\cutradius)}
  \pgfmathsetmacro{\ringcapangle}{acos(\capx/\cutx)}
  \pgfmathsetmacro{\ringbandangle}{acos(\bandx/\cutx)}
  \pgfmathsetmacro{\originalcapx}{sqrt(\lensradius^2-
    (\lensoffset/\viewheight)^2)}
  \pgfmathsetmacro{\originalcapy}{\lensoffset*\viewdepth^2/\viewheight}
  \pgfmathsetmacro{\originalcapangle}{acos(\originalcapx/\lensradius)}
  \pgfmathsetmacro{\originalringangle}{acos(\originalcapx/\rimradius)}
  \pgfmathsetmacro{\rimvisibleangle}{acos(\cutradius/\rimradius)}
  \pgfmathsetmacro{\projectedcutangle}{atan2(\cutx,\ringheight)}
  \draw[ttzzqq!40,line width=.65pt,dash pattern=on 2pt off 3pt]
    (\originalcapx,\originalcapy)
    arc[start angle=\originalcapangle,end angle={180-\originalcapangle},
        radius=\lensradius cm]
    arc[start angle={180-\originalringangle},end angle={180+\originalringangle},
        x radius=\rimradius cm,y radius=.4756cm]
    arc[start angle={180+\originalcapangle},end angle={360-\originalcapangle},
        radius=\lensradius cm]
    arc[start angle={360-\originalringangle},end angle={360+\originalringangle},
        x radius=\rimradius cm,y radius=.4756cm];
  \draw[ttzzqq!30,line width=.55pt,dash pattern=on 2pt off 3pt]
    plot[domain=0:\rimvisibleangle,samples=25,variable=\t]
    ({\rimradius*cos(\t)},{.4756*sin(\t)});
  \draw[ttzzqq!30,line width=.55pt,dash pattern=on 2pt off 3pt]
    plot[domain={180-\rimvisibleangle}:180,samples=25,variable=\t]
    ({\rimradius*cos(\t)},{.4756*sin(\t)});
  \fill[ttzzqq,fill opacity=.12]
    (\capx,\capy)
    arc[start angle=\capangle,end angle={180-\capangle},
        radius=\lensradius cm]
    arc[start angle={180-\ringcapangle},end angle={180-\ringbandangle},
        x radius=\cutx cm,y radius=\ringdepth cm]
    arc[start angle={180-\bandangle},end angle={180+\bandangle},
        radius=\cutradius cm]
    arc[start angle={180+\ringbandangle},end angle={180+\ringcapangle},
        x radius=\cutx cm,y radius=\ringdepth cm]
    arc[start angle={180+\capangle},end angle={360-\capangle},
        radius=\lensradius cm]
    arc[start angle={360-\ringcapangle},end angle={360-\ringbandangle},
        x radius=\cutx cm,y radius=\ringdepth cm]
    arc[start angle={360-\bandangle},end angle={360+\bandangle},
        radius=\cutradius cm]
    arc[start angle=\ringbandangle,end angle=\ringcapangle,
        x radius=\cutx cm,y radius=\ringdepth cm] -- cycle;
  \foreach \height in {\ringheight,-\ringheight}{
    \draw[ttzzqq,line width=.65pt,dash pattern=on 1.5pt off 2.5pt]
      (\cutx,\height)
      arc[start angle=0,end angle=180,
          x radius=\cutx cm,y radius=\ringdepth cm];
    \draw[ttzzqq,line width=.8pt]
      (-\cutx,\height)
      arc[start angle=180,end angle=360,
          x radius=\cutx cm,y radius=\ringdepth cm];
  }
  \draw[ttzzqq,line width=1pt]
    (\bandx,\bandy)
    arc[start angle=\ringbandangle,end angle=\ringcapangle,
        x radius=\cutx cm,y radius=\ringdepth cm]
    arc[start angle=\capangle,end angle={180-\capangle},
        radius=\lensradius cm]
    arc[start angle={180-\ringcapangle},end angle={180-\ringbandangle},
        x radius=\cutx cm,y radius=\ringdepth cm];
  \draw[ttzzqq,line width=1pt]
    (-\bandx,-\bandy)
    arc[start angle={180+\ringbandangle},end angle={180+\ringcapangle},
        x radius=\cutx cm,y radius=\ringdepth cm]
    arc[start angle={180+\capangle},end angle={360-\capangle},
        radius=\lensradius cm]
    arc[start angle={360-\ringcapangle},end angle={360-\ringbandangle},
        x radius=\cutx cm,y radius=\ringdepth cm];
  \draw[uququq,line width=1.05pt]
    (-\bandx,\bandy)
    arc[start angle={180-\bandangle},end angle={180+\bandangle},
        radius=\cutradius cm];
  \draw[uququq,line width=1.05pt]
    (\bandx,-\bandy)
    arc[start angle={360-\bandangle},end angle={360+\bandangle},
        radius=\cutradius cm];

  \draw[ttzzqq!40,line width=.65pt,dash pattern=on 2pt off 3pt]
    (-\rimradius,0)
    arc[start angle=180,end angle=360,
        x radius=\rimradius cm,y radius=.4756cm];

  \draw[uququq!55,line width=.6pt,dash pattern=on 2pt off 2pt]
    (0,0) -- (0,{\lensheight*\viewheight-.12});
  \node[above] at (0,{\lensradius-\lensoffset*\viewheight+.08}) {$e$};
  \draw[uququq,line width=.7pt] (0,0) -- (\cutx,\ringheight);
  \node[below,rotate={90-\projectedcutangle}]
    at ({.62*\cutx},{.62*\ringheight}) {$\ell_s(T)$};
  \draw[uququq,line width=.65pt] (0,.58)
    arc[start angle=90,end angle={90-\projectedcutangle},radius=.58cm];
  \node at ({.8*sin(\projectedcutangle/2)},{.8*cos(\projectedcutangle/2)}) {$T$};
  \node[left] at (0,{.67*\lensheight*\viewheight}) {$s$};
  \fill[uququq] (0,0) circle[radius=1.2pt];
  \node[below left,inner sep=3pt] at (0,0) {$o$};
  \fill[uququq] (\cutx,\ringheight) circle[radius=1.2pt];

  \node at (-.95,-.65) {$\mathscr V(s)$};
\end{tikzpicture}
\caption{The radial cutoff of a lens, with radial function $\ell_s(d_{\pm e,T})$. The original lens is shown dashed. For $T=T(s)$, the volume of the shaded body is $\sM(s)$.}
\label{Fig:Truncation}
\end{figure}

The proof of Theorem~\ref{Thm:KeyA} below relies on the following monotonicity property. 

\begin{lemma}
\label{Lem:Rr}
For every $s \in (0,1)$, the function
\[
\theta \mapsto \Rr_s(\theta) := \frac{V_s'(\theta)}{A_s'(\theta)}
\]
is monotonically increasing on $[0, \pi/2]$. (Here, primes denote differentiation with respect to $\theta$.)
\end{lemma}
\begin{proof}
We have that 
$$
\frac{d}{d\theta} \left(\cos \phi_s(\theta)\right) = -\frac{(1-s)^2 \sin \theta \cos \theta}{\cos \phi_s(\theta)}, \qquad \ell_s'(\theta) =\frac{(1-s)\sin\theta}{\cos \phi_s(\theta)}\ell_s(\theta)
$$
and
\begin{equation*}
    \begin{aligned}
A_s'(\theta)&=\frac{(1-s)\sin\theta \cdot \ell_s(\theta)^{n-1}}{\left(\cos \phi_s(\theta)\right)^3} \bigl((n-1)\cos \phi_s(\theta) + (1-s)\cos\theta\bigr)\\
&=\frac{(1-s)\sin\theta \cdot \ell_s(\theta)^{n-1}}{\left(\cos \phi_s(\theta)\right)^3} \bigl(n\cos \phi_s(\theta) - \ell_s(\theta)\bigr)>0.\\
    \end{aligned}
\end{equation*}
Therefore, for $\theta \in [0, \pi/2]$,
\[
\Rr_s(\theta) = \frac{V_s'(\theta)}{A_s'(\theta)}
=\frac{\ell_s(\theta)(\cos \phi_s(\theta))^2}
       {n \cos \phi_s(\theta)-\ell_s(\theta)},
\]
(where the values at $0$ and $\pi/2$ are defined by continuity). Hence, 
$$
\Rr_s'(\theta)
=\frac{(1-s)\ell_s(\theta)^2\sin\theta}
       {\bigl(n \cos \phi_s(\theta)- \ell_s(\theta)\bigr)^2}
 \bigl(n \cos \phi_s(\theta)+2(1-s)\cos\theta\bigr)>0.
$$
\end{proof}

\vspace{-5mm}
The following theorem is a key estimate in the proof.

\begin{theorem}[Key Comparison]
\label{Thm:KeyA}
Let  $K \subset \R^n$ be a 1-convex polytope whose facets all touch its inscribed ball with the inradius $r = r(K)$. If the cutoff $T_0 = T_0(r)$ is chosen so that \eqref{Eq:Area} holds with $s=r$, then
$$
|K|\geqslant \sM(r) =\int_{\mathbb S^{n-1}}V_r(d_{\pm e,T_0}(\omega))\,d\omega.
$$    
\end{theorem}

\begin{proof}
For a measurable function $h:\mathbb S^{n-1}\to[0,\pi/2]$, put $F_h(s):=|\{\omega \in \S^{n-1}:h(\omega)>s\}|$ to be the measure of the superlevel sets of $h$. Then for every $C^1$-smooth function $H$ on $[0,\pi/2]$, the fundamental theorem of calculus gives
$$
H(h(\omega))
=H(0)
 +\int_0^{h(\omega)}H'(s) \,ds.
$$
Integrating the above identity over $\mathbb{S}^{n-1}$, we obtain
$$
\int_{\mathbb S^{n-1}}H(h(\omega))\,d\omega
=H(0)|\mathbb S^{n-1}|
 +\int_0^{\pi/2}H'(s)F_h(s)\,ds
$$
where we used Fubini's theorem in the last integral.

Applying this identity to $H=V_r$ and $h$ equals to $d_Q$ or $d_{\pm e, T_0}$:
\begin{equation}
    \begin{aligned}
    \label{Eq:1}
    |K| - \int_{\S^{n-1}}V_r(d_{\pm e, T_0}(\omega)) d\omega &= \int_{\S^{n-1}}V_r(d_Q(\omega)) d\omega - \int_{\S^{n-1}}V_r(d_{\pm e, T_0}(\omega))d\omega\\
    &=\int_0^{\pi/2}V'_r(s)\left(F_{d_Q}(s)-F_{d_{\pm e, T_0}}(s)\right)ds\\
    &=\int_0^{\pi/2}\Rr_r(s) A'_r(s)\left(F_{d_Q}(s)-F_{d_{\pm e, T_0}}(s)\right)ds
    \end{aligned}
\end{equation}
where we used notation from Lemma~\ref{Lem:Rr}.
Likewise, applying it with $H =A_r$ and using that $\sA(r,T_0(r)) = |\partial K|$, we get
\begin{equation}
    \begin{aligned}
    \label{Eq:2}
    0 = |\partial K| - \int_{\S^{n-1}}A_r(d_{\pm e, T_0}(\omega))d\omega &= \int_{\S^{n-1}}A_r(d_Q(\omega)) d\omega - \int_{\S^{n-1}}A_r(d_{\pm e, T_0}(\omega))d\omega\\
    &=\int_0^{\pi/2}A'_r(s)\left(F_{d_Q}(s)-F_{d_{\pm e, T_0}}(s)\right)ds.    
    \end{aligned}
\end{equation}
Multiplying \eqref{Eq:2} by $\Rr_r(T_0)$ and subtracting it from \eqref{Eq:1} gives
\[
\begin{aligned}
|K|-\int_{\mathbb S^{n-1}}V_r(d_{\pm e, T_0}(\omega))\,d\omega
=\int_0^{\pi/2}A_r'(s)
 \bigl(\Rr_r(s)-\Rr_r(T_0)\bigr)
 \bigl(F_{d_Q}(s)-F_{d_{\pm e, T_0}}(s)\bigr)\,ds.
\end{aligned}
\]

If $s < T_0$, then $F_{d_{\pm e, T_0}}(s)=F_{d_{\pm e}}(s) \ge F_{d_Q}(s)$ by \eqref{Eq:Comp0}. If $s>T_0$, then $F_{d_{\pm e, T_0}}(s)=0 \le F_{d_Q}(s)$. By Lemma~\ref{Lem:Rr}, $\Rr_r$ is increasing, 
so $\Rr_r(s)-\Rr_r(T_0)$ has the same sign as $F_{d_Q}(s)-F_{d_{\pm e, T_0}}(s)$ for each $s < T_0$ and $s>T_0$. Thus, the integrand is always nonnegative on both intervals. This proves the theorem.
\end{proof}

So far, we have that $|K| \ge \sM(r)$ by Theorem~\ref{Thm:KeyA} and $\sM(r_0) = |L|$ by \eqref{Eq:Final}. To compare $\sM(r)$ and $\sM(r_0)$, we now show that $\sM$ is strictly increasing:
\begin{lemma}
\label{Lem:M}
The function $s \mapsto \sM(s)$ is strictly increasing for $s \in [r_0,r]$.
\end{lemma}

\begin{proof}
Let $r_0<s<r$ and write $T=T(s) \in (0, \pi/2)$. By rotational symmetry, we can write $\sA, \sM$ as follows:
\begin{equation}
 \label{Eq:SimpleA}
\sA(s,T) = 2|\mathbb S^{n-2}| \left(\int_{0}^T A_s(\theta)\sin^{n-2}\theta \,d\theta + A_s(T)\int_{T}^{{\pi}/{2}} \sin^{n-2}\theta \,d\theta\right), 
\end{equation}
\begin{equation}
 \label{Eq:SimpleV}
\sM(s) = 2|\mathbb S^{n-2}| \left(\int_{0}^{T} V_s(\theta)\sin^{n-2}\theta \,d\theta + V_s(T)\int_{T}^{{\pi}/{2}} \sin^{n-2}\theta \,d\theta\right).
\end{equation}
Differentiating~\eqref{Eq:SimpleA} with respect to $T$ and using that $T \in (0, \pi/2)$, we get
\[
\partial_T\sA(s,T)
=2|\mathbb S^{n-2}| A_s'(T)
\int_T^{\pi/2}\sin^{n-2}\theta\,d\theta>0,
\]
Then the implicit function theorem implies that $T(s)$ is differentiable for $s \in (r_0, r)$. Here, $A_s'(\theta) = \partial A_s(\theta) /\partial \theta $, and $\partial_sA_s(\theta) = \partial A_s(\theta) /\partial s$.

Differentiating \eqref{Eq:Area} with respect to $s$, we obtain
\begin{equation*}
\partial_s\sA(s,T)+\partial_T\sA(s,T) \cdot T'(s)=0  
\end{equation*}
which together with~\eqref{Eq:SimpleA} implies
\begin{equation}\label{Eq:A}
    0 = \int_{0}^T \partial_s A_s(\theta) \sin^{n-2}\theta \, d\theta + \left(\partial_s A_s(T) + A_s'(T) \cdot T'(s) \right)\int_{T}^{\pi/2}  \sin^{n-2}\theta \, d\theta.   
\end{equation}
Similarly, differentiating \eqref{Eq:SimpleV}, we get
\begin{equation}
\label{Eq:B}
    \frac{\sM'(s)}{2|\mathbb S^{n-2}|} = \int_{0}^T \partial_s V_s(\theta) \sin^{n-2}\theta \, d\theta + \left(\partial_s V_s(T) + V_s'(T) \cdot T'(s) \right)\int_{T}^{\pi/2}  \sin^{n-2}\theta \, d\theta.
\end{equation}
To eliminate terms with $T'(s)$, we multiply \eqref{Eq:A} by $\Rr_s(T) = V'_s(T) / A'_s(T)$ and subtract it from \eqref{Eq:B}, so we obtain
\[
\begin{aligned}
\frac{\sM'(s)}{2|\mathbb S^{n-2}|} ={}&\int_0^T \left(\partial_s V_s(\theta)
-\Rr_s(T)\partial_s A_s(\theta)\right) \sin^{n-2}\theta\,d\theta\\
&+\left(\partial_s V_s(T) -\Rr_s(T)\partial_s A_s(T)\right)\int_T^{\pi/2}\sin^{n-2}\theta\,d\theta.
\end{aligned}
\]
First, we simplify the integrals in the last identity. We have:
\[
\partial_s\bigl(\cos\phi_s(\theta)\bigr) =\frac{(1-s)\sin^2\theta}{\cos\phi_s(\theta)},
\qquad \partial_s\ell_s(\theta) =\cos\theta+ \frac{(1-s)\sin^2\theta}{\cos\phi_s(\theta)}.
\]
Together with the formulas for the derivatives with respect to
$\theta$, these give
\[
\begin{aligned}
\bigl(\partial_s A_s(\theta)\bigr)\sin^{n-2}\theta &=\frac{d}{d\theta}
\left(
\frac{\ell_s(\theta)^{n-2}\sin^{n-1}\theta}
     {\cos\phi_s(\theta)}
\right),\\
\bigl(\partial_s V_s(\theta)\bigr)\sin^{n-2}\theta &=\frac{1}{n-1}\frac{d}{d\theta} \left(\ell_s(\theta)^{n-1}\sin^{n-1}\theta\right).
\end{aligned}
\]
Integrating these identities over $[0, T]$, gives
\begin{equation}
\label{Eq:Volll}
\begin{aligned}
\frac{\sM'(s)}{2|\mathbb S^{n-2}|}
={}& \left(\frac{\ell_s(T)^{n-1}}{n-1} -\Rr_s(T)\frac{\ell_s(T)^{n-2}}{\cos\phi_s(T)}\right)\sin^{n-1}T\\
&+\left(\partial_s V_s(T) -\Rr_s(T)\partial_s A_s(T)\right) \int_T^{\pi/2}\sin^{n-2}\theta\,d\theta.
\end{aligned}
\end{equation}
Using the formula for $\Rr_s(T)$ and the identity
$\ell_s(T)=\cos\phi_s(T) - (1-s)\cos T$, we compute
\[
\begin{aligned}
\partial_s V_s(T)-\Rr_s(T)\partial_s A_s(T)
&=
\frac{\ell_s(T)^{n-1}}
     {n\cos\phi_s(T)-\ell_s(T)}
\left(
(1-s)\cos T\,\partial_s\ell_s(T)
+\ell_s(T)\partial_s\bigl(\cos\phi_s(T)\bigr)
\right)\\
&=
\frac{(1-s)\ell_s(T)^{n-1}}
     {n\cos\phi_s(T)-\ell_s(T)}.
\end{aligned}
\]
Indeed, the expression in parentheses equals
\[
(1-s)\cos^2T+
\frac{(1-s)\sin^2T}{\cos\phi_s(T)}
\bigl((1-s)\cos T+\ell_s(T)\bigr)=1-s.
\]
Similarly,
\[
\frac{\ell_s(T)^{n-1}}{n-1}
-\Rr_s(T)\frac{\ell_s(T)^{n-2}}{\cos\phi_s(T)}
=
\frac{(1-s)\ell_s(T)^{n-1}\cos T}
     {(n-1)\bigl(n\cos\phi_s(T)-\ell_s(T)\bigr)}.
\]
Substituting these identities into~\eqref{Eq:Volll} yields
\[
\frac{\sM'(s)}{2|\mathbb S^{n-2}|}=
\frac{(1-s)\ell_s(T)^{n-1}}
     {n\cos\phi_s(T)-\ell_s(T)}
\left(
\int_T^{\pi/2}\sin^{n-2}\theta\,d\theta
+\frac{\cos T\,\sin^{n-1}T}{n-1}
\right)>0.
\]
The strict inequality follows from $0<s<1$ and $0<T<\pi/2$. Continuity at the endpoints completes the proof.
\end{proof}

We can now complete the proof in Case \ref{It:Case1}. Using Theorem~\ref{Thm:KeyA}, Lemma~\ref{Lem:M}, and equation \eqref{Eq:Final}, we get
\begin{equation}
    \label{Eq:FinalClaim}
|K|\geqslant \sM(r)\geqslant \sM(r_0)=|L| \qquad{\text{with equality if and only if} \qquad r=r_0}.
\end{equation}
If $|K| = |L|$, then $r = r_0$ so the equality case in Theorem~\ref{thm:reverse-inradius} gives that $K$ should be a $1$-convex lens. On the other hand, equality holds when $K$ is a $1$-convex lens. This completes the proof of the Main Theorem in Case \ref{It:Case1}.

\subsection{Case \ref{It:Case2}: General case of $1$-convex polytopes}
\label{SSec:NoTouch}

Let $K$ be a $1$-convex polytope such that at least one facet of $K$ does not touch the inscribed ball $B$. We adapt the proof of Subsection~\ref{SSec:Touch} to this case. 

By Carath\'eodory's theorem and maximality of $B$, we can choose a subset of $\partial K \cap \partial B$ of at most $n+1$ touching points whose convex hull contains the origin. Let $Q\subset \mathbb S^{n-1}$ be the radial projections of these points onto the unit sphere. Then $Q$ is not contained in any open hemisphere. The corresponding tangent supporting balls contain $K$. Therefore, if $\rho_K \colon \S^{n-1} \to (0, +\infty)$ is the radial function of $K$, then
$$
r\leqslant\rho_K(\omega) \le \ell_r(d_Q(\omega)), \qquad \omega \in \S^{n-1}.
$$

We can now use the comparison construction of \cite{BorDr13, DrBla}. For every $\omega \in \S^{n-1}$, let $\Theta(\omega) \in \S^1$ be the unique angle such that 
$$
\rho_K(\omega)=\ell_r(\Theta(\omega)).
$$
By the Radial Angle Comparison theorem~\cite{BorDr13, DrBla}, we have
\begin{equation}
   \label{Eq:RAC1}
    \cos \phi_K(\omega) \ge \cos \phi_r(\Theta(\omega)), \qquad \omega \in \S^{n-1}.   
\end{equation}
Observe that $\Theta(\omega)\leqslant d_Q(\omega)$ for all $\omega \in \S^{n-1}$. Therefore, for every $s \in [0, \pi/2]$, Lemma~\ref{lem:spherical-comparison} gives 
\begin{equation}
\label{Eq:Comp1}
|\{\omega \in \S^{n-1} : \Theta(\omega)>s\}| \le |\{\omega \in \S^{n-1} : d_Q(\omega) > s\}| \le |\{\omega \in \S^{n-1} : d_{\pm e}(\omega) > s\}|.
\end{equation}

Furthermore, by \eqref{Eq:AVSupport} and \eqref{Eq:RAC1},
\begin{equation}
\label{Eq:A1}
|\partial K| = \int_{\S^{n-1}} \frac{\rho_K(\omega)^{n-1}}{\cos \phi_K(\omega)} d\omega \le \int_{\S^{n-1}} \frac{\ell_r(\Theta(\omega))^{n-1}}{\cos \phi_r(\Theta(\omega))} d\omega = \int_{\S^{n-1}} A_r(\Theta(\omega)) d\omega.  
\end{equation}
Finally, by \eqref{Eq:AVSupport} and the definition of $\Theta$, we have
\begin{equation*}
\label{Eq:V1}
|K| = \int_{\S^{n-1}} V_r(\Theta(\omega))d\omega.    
\end{equation*}

The proof now continues as in Case \ref{It:Case1} except at one estimate. Namely, we argue as in Theorem~\ref{Thm:KeyA}, replacing the function $ d_Q$ with $\Theta$, and keeping the cutoff $T_0$ that is determined by~\eqref{Eq:Area}. Then the volume identity~\eqref{Eq:1} remains unchanged
\begin{equation}
    \label{Eq:A3}
|K| - \int_{\S^{n-1}} V_r(d_{\pm e, T_0}(\omega)) d \omega = \int_0^{\pi/2} \Rr_r(s) A_r'(s)\big(F_\Theta(s) - F_{d_{\pm e, T_0}}(s)\big)ds.    
\end{equation}
The area identity \eqref{Eq:2} now becomes an inequality due to \eqref{Eq:A1}:
\begin{equation*}
\begin{aligned}
0 = |\partial K| - \int_{\S^{n-1}}A_r(d_{\pm e, T_0}(\omega))d\omega &\le \int_{\S^{n-1}}A_r(\Theta(\omega)) d\omega - \int_{\S^{n-1}}A_r(d_{\pm e, T_0}(\omega))d\omega\\
    &=\int_0^{\pi/2}A'_r(s)\left(F_{\Theta}(s)-F_{d_{\pm e, T_0}}(s)\right)ds.        
\end{aligned}    
\end{equation*}

Multiplying the last inequality by $\Rr_r(T_0)$ and subtracting from \eqref{Eq:A3}, we obtain the following estimate:
\[
|K|-\int_{\S^{n-1}} V_r(d_{\pm e, T_0}(\omega)) d \omega \ge \int_0^{\pi/2}A_r'(s)  \bigl(\Rr_r(s)-\Rr_r(T_0)\bigr) \bigl(F_\Theta(s)-F_{d_{\pm e, T_0}}(s)\bigr)\,ds.
\]
The conclusion is then again similar to that of Theorem~\ref{Thm:KeyA}. If $s<T_0$, then $F_{d_{\pm e, T_0}}(s)=F_{d_{\pm e}}(s) \ge F_{\Theta}(s)$ by \eqref{Eq:Comp1}. If $s>T_0$, then $F_{d_{\pm e, T_0}}(s)=0 \le F_{\Theta}(s)$. Since $\Rr_r$ is increasing by Lemma~\ref{Lem:Rr}, the integrand is always nonnegative on both intervals. 

Thus, $|K|\geqslant \sM(r)$, and $\sM(r) \geqslant \sM(r_0) =|L|$ by Lemma~\eqref{Lem:M}.
As in Case~\ref{It:Case1}, equality forces $r = r_0$, and the equality case in Theorem~\ref{thm:reverse-inradius} implies that $K$ is a $1$-convex lens.

\subsection{General case of $1$-convex bodies}
\label{SSec:General}

Let $K$ be an arbitrary $1$-convex body. By \cite[Proposition 2.5]{DT}, there exists a decreasing sequence $(K_j)$ of $1$-convex polytopes converging to $K$ in the Hausdorff metric and $r(K_j) \to r(K) = r$. For such a sequence, we can construct the decreasing sequence $(L_j)$ of $1$-convex lenses centered at the origin so that $|\partial L_j| = |\partial K_j|$, the sequence converges to $L$ in the Hausdorff metric, and $r(L_j) \to r(L)=r_0$. Since, by the construction above, the inequality holds for $1$-convex polytopes, we have $|K_j| \ge  |L_j|$. Hence, by passing to the limit, we get $|K| \ge |L|$, which proves the inequality.   

If $r = r_0$, then $K$ is a $1$-convex lens by the rigidity part of Theorem~\ref{thm:reverse-inradius}. Assume $r > r_0$. Note that $r(L_j) \to r_0$, and $r(K_j) \to r$. Thus, we can find $s$ such that $s \in (r(L_j), r(K_j))$ for all sufficiently large $j$. Writing \eqref{Eq:FinalClaim} for $K_j$ and $L_j$ and $s$ as above, we get
\[
|K_j| \ge \sM_j(r(K_j)) > \sM_j(s)> \sM_j(r(L_j))=|L_j|,
\]
where $\sM_j$ is the corresponding quantity defined for $L_j$. Passing to the limit as $j \to \infty$, since $s$ is independent of $j$ and $s \in (r_0, r)$, we obtain
\[
|K| \ge \sM(r) > \sM(s) > \sM(r_0)=|L|.
\]
This proves the strict inequality, thus finishing the proof.


\section{Proof of the Main Theorem in spherical and hyperbolic spaces}
\label{sec:model-proof}

We extend the argument of Section~\ref{sec:euclidean-proof} to the model spaces $M^n(c)$ with $c\ne0$. Let $\lambda>0$, and define the generalized trigonometric functions
\[
\sn_c(t):=
\begin{cases}
\dfrac{\sin(\sqrt c\,t)}{\sqrt c},&c>0,\\[2mm]
t,&c=0,\\[1mm]
\dfrac{\sinh(\sqrt{-c}\,t)}{\sqrt{-c}},&c<0,
\end{cases}
\qquad \cn_c(t):=\sn_c'(t).
\]
Recall that $\cn_c'=-c\sn_c$ and $\cn_c^2+c\sn_c^2=1$.

We first prove the result for $\lambda$-convex polytopes. The proof for arbitrary $\lambda$-convex bodies goes by approximation as in Subsection~\ref{SSec:General} of the Euclidean case, and hence is omitted.

Let $K\subset M^n(c)$ be a $\lambda$-convex polytope, and let $B :=B(o,r)$ be an inscribed ball for $K$ with $r = r(K)$. Note that in the case $c < 0$, it might not be unique (see the discussion in \cite{DT}). 

As before, let $L$ be a $\lambda$-convex lens with $|\partial L|=|\partial K|$ and set $r_0:=r(L)$. The reverse inradius inequality (Theorem~\ref{thm:reverse-inradius}) gives $r_0\le r$.

When the $\lambda$-spheres are geodesic spheres, denote their radius by $R_\lambda$, characterized by $\cn_c(R_\lambda)=\lambda\sn_c(R_\lambda)$. If $r=R_\lambda$, then $K=B(o,R_\lambda)$, and the inequality that we want to prove is immediate. We therefore, for the rest of the section, assume $r<R_\lambda$ in this case. 

For equidistants, that is, when $c<0$ and $0<\lambda<\sqrt{-c}$, define the characteristic radius
\[
\rr_\lambda:=\frac1{\sqrt{-c}}\operatorname{arctanh}\frac{\lambda}{\sqrt{-c}},
\]
see the discussion in \cite[Section 4]{DT}. 

\begin{remark}
Consider two equidistant $\lambda$-spheres that are tangent at antipodal points to the geodesic ball of radius $t$ centered at the origin, and whose both associated $\lambda$-balls contain this geodesic ball. Their intersection is compact if $t < \rr_\lambda$ when it is a $\lambda$-convex lens of inradius $t$. Thus, geometrically, $\rr_\lambda$ is the supremum of inradii $t$ of $\lambda$-convex lenses in this case. In particular, there is no $\lambda$-convex lens with such an inradius when $c < 0$, $\lambda \in (0, \sqrt{-c})$ and  $r \ge \rr_\lambda$ so this regime requires separate treatment below.
\end{remark}

Let $\Ss_\lambda$ be a $\lambda$-sphere tangent to $B$ at $p$ and whose $\lambda$-ball contains $B$. Identify the unit tangent sphere at $o$ with $\S^{n-1}$, and denote the direction of $op$ by $\omega_0$. For every direction $\omega$ with $\theta=\dist_{\S^{n-1}}(\omega_0,\omega)\le\pi/2$, let $x$ be the first intersection point of the geodesic ray from $o$ in the direction $\omega$ that meets $\Ss_\lambda$, whenever it exists.
Let $\ell_r(\theta):=\dist_{M^n(c)}(o,x)$ be the radial function of $\Ss_\lambda$, and let $\phi_r(\theta)$ be its radial angle defined as the angle between the outward radial direction and the outer unit normal to $\Ss_\lambda$ at $x$.

For $s\in[r_0,r]$, put
\[
a(s):=\cn_c(s)-\lambda\sn_c(s),
\qquad b(s):=\lambda\cn_c(s)+c\sn_c(s).
\]
Thus $a'=-b$ and $b'=ca$. We have $a(s)>0$ for all $s \in [r_0, r]$. In the case of geodesic spheres (i.e., when $c > 0$, or $c < 0$ and $\lambda > \sqrt{-c}$) or horospheres (i.e., when $c<0$ and $\lambda = \sqrt{-c}$), it follows that $0<a(s)<1$ and $b(s)>0$. For the case of geodesic spheres
\[
a(s)=\frac{\sn_c(R_\lambda-s)}{\sn_c(R_\lambda)}.
\]

In the equidistant case (i.e., when $c<0$ and $\lambda \in (0, \sqrt{-c})$), write $c=-k^2<0$. Since $L$ is compact, $r_0<\rr_\lambda$. For every $s\in[r_0,r]$,
\[
a(s)=\frac{\sqrt{k^2-\lambda^2}}k\cosh(k(s-\rr_\lambda)),
\qquad
b(s)=-\sqrt{k^2-\lambda^2}\sinh(k(s-\rr_\lambda)).
\]
Thus $0<a(s)<1$ and $b(s)>0$ for $s<\rr_\lambda$, while $b(\rr_\lambda)=0$ and $b(s)<0$ for $s>\rr_\lambda$.

A direct computation for $\Ss_\lambda$ (see, e.g., \cite{BorDr13,DrPhD}) gives
\begin{equation}
\label{Eq:ModelAngle}
\sin\phi_s(\theta)=a(s)\sin\theta,
\qquad
\cos\phi_s(\theta)=\sqrt{1-a(s)^2\sin^2\theta}.
\end{equation}
The radial function $\ell_s$ also satisfies the following identities which can be obtained by direct computation
\begin{equation}
\label{Eq:ModelIdentities}
\begin{aligned}
\cn_c(\ell_s(\theta))\cos\phi_s(\theta)
-\lambda\sn_c(\ell_s(\theta))&=a(s)\cos\theta,\\
\lambda\cn_c(\ell_s(\theta))
+c\sn_c(\ell_s(\theta))\cos\phi_s(\theta)&=b(s),\\
a(s)\cn_c(\ell_s(\theta))\cos\theta
+b(s)\sn_c(\ell_s(\theta))&=\cos\phi_s(\theta).
\end{aligned}
\end{equation}
If $b(s)>0$, then $\ell_s(\theta)$ is finite and strictly increasing for $\theta \in [0,\pi/2]$. In the spherical case, the values of $\ell_s(\theta)$  on  $[0,\pi/2]$ are less than $R_\lambda<\pi/(2\sqrt c)$. In the equidistant case with $s\ge\rr_\lambda$, the domain of $\ell_s$ is instead
\[
0\le\theta<\theta_\infty(s):= \arccos\left(\frac{-b(s)}{ka(s)}\right) =\arcsin\frac1{\cosh(k(s-\rr_\lambda))}.
\]
On this interval $\ell_s$ is strictly increasing, and $\ell_s(\theta)\to\infty$ as $\theta\to\theta_\infty(s)$. All the radial identities above are understood on this domain. 

The corresponding area and volume densities~\eqref{Eq:AreaVol} are
\begin{equation}
\label{Eq:ModelAreaVol}
A_s(\theta):=\frac{\sn_c(\ell_s(\theta))^{n-1}}{\cos\phi_s(\theta)},
\qquad
V_s(\theta):=\int_0^{\ell_s(\theta)}\sn_c(t)^{n-1}\,dt.
\end{equation}
When $c=0$ and $\lambda=1$, they recover Euclidean densities~\eqref{Eq:AreaVol}.

Now we return to $K$. By the maximality of $B$ and Carath\'eodory's theorem, choose at most $n+1$ points of $\partial K\cap\partial B$ whose radial directions form a set $Q\subset\S^{n-1}$ not contained in any open hemisphere. By Theorem~\ref{Thm:Bla}, the supporting  $\lambda$-balls at these points contain $K$. Therefore, if $\rho_K$ is the radial function of $K$ with respect to $o$, then 
\[
r\le\rho_K(\omega)\le\ell_r(d_Q(\omega))
\]
for every $\omega$ where $d_Q(\omega)$ is in the domain of $\ell_r$. Here we use notation from Section~\ref{subsec:spherical-lemma}.

As in Section~\ref{SSec:NoTouch}, define $\Theta:\S^{n-1}\to[0,\pi/2]$ by
\begin{equation}
\label{Eq:ModelTheta}
\rho_K(\omega)=\ell_r(\Theta(\omega)).
\end{equation}
Thus $\Theta\le d_Q$. In the equidistant case with $r\ge\rr_\lambda$, the function $\ell_r$ maps $[0,\theta_\infty(r))$ onto $[r,\infty)$, so $\Theta$ is still well-defined. Also, if $d_Q(\omega)\ge\theta_\infty(r)$, then $\Theta(\omega)<\theta_\infty(r)\le d_Q(\omega)$ automatically. Moreover, compactness of $K$ gives
\[
\max_{\omega \in \S^{n-1}}\Theta(\omega)<\theta_\infty(r).
\]
Since $\theta_\infty(s)\ge\theta_\infty(r)$ for $\rr_\lambda\le s\le r$, the same $\Theta$ takes values in the domain of $\ell_s$ for every such $s$.

By the Radial Angle Comparison theorem~\cite{BorArea, BorDr13, BDr15} (see \cite{DrPhD} for an overview), for every $\omega \in \S^{n-1}$ we have
\begin{equation}
\label{Eq:ModelRAC}
\phi_K(\omega)\le\phi_r(\Theta(\omega)),
\qquad\text{or equivalently}\qquad
\cos\phi_K(\omega)\ge\cos\phi_r(\Theta(\omega)).
\end{equation}
Note that in the references above, the result is proven assuming $\partial K$ is smooth. But in model spaces, it can be easily extended to the piecewise smooth case.

Lemma~\ref{lem:spherical-comparison}, together with $\Theta\le d_Q$, now yields for every $t\in[0,\pi/2]$:
\begin{equation}
\label{Eq:ModelSuperlevels}
|\{\omega \in \S^{n-1} : \Theta(\omega)>t\}|\le|\{\omega \in \S^{n-1} : d_Q(\omega)>t\}|\le|\{\omega \in \S^{n-1} : d_{\pm e}(\omega)>t\}|.
\end{equation}
Finally, the radial area and volume formulas give
\begin{equation}
\label{Eq:ModelAVComparison}
\begin{aligned}
|\partial K|&=\int_{\S^{n-1}}
\frac{\sn_c(\rho_K(\omega))^{n-1}}{\cos\phi_K(\omega)}\,d\omega
\le\int_{\S^{n-1}}A_r(\Theta(\omega))\,d\omega,\\
|K|&=\int_{\S^{n-1}}\int_0^{\rho_K(\omega)}\sn_c(t)^{n-1}\,dt\,d\omega
=\int_{\S^{n-1}}V_r(\Theta(\omega))\,d\omega.
\end{aligned}
\end{equation}
Thus we recover the same form of the area and volume expressions as in the general Euclidean case in Section~\ref{SSec:NoTouch}.

As in Section~\ref{SSec:Touch}, the proof of Theorem~\ref{Thm:ModelKeyA} below relies on the following monotonicity property.

\begin{lemma}
\label{Lem:ModelRr}
For every $s \in [r_0,r]$ with $b(s)\ge0$, the function
\[
\theta\longmapsto\Rr_s(\theta):=\frac{V_s'(\theta)}{A_s'(\theta)}
\]
is strictly increasing on $[0,\pi/2]$ if $b(s)>0$, and on $[0,\pi/2)$ if $b(s)=0$, where the value at $0$ is defined by continuity. Moreover, $A_s'(\theta)>0$ for every positive $\theta$ in the domain. (Here and below, primes denote differentiation with respect to $\theta$.)
\end{lemma}
\begin{proof}
Throughout this computation, $\ell_s$ and $\phi_s$ are evaluated at $\theta$, and we suppress this argument for brevity. From~\eqref{Eq:ModelAngle} and \eqref{Eq:ModelIdentities}, we have
\[
\ell_s'=\frac{a(s)\sn_c(\ell_s)\sin\theta}{\cos\phi_s},
\qquad
(\cos\phi_s)'=-\frac{a(s)^2\sin\theta\cos\theta}{\cos\phi_s}.
\]
Differentiating the densities~\eqref{Eq:ModelAreaVol} now gives
\[
\begin{aligned}
V_s'&=\frac{a(s)\sn_c(\ell_s)^n\sin\theta}{\cos\phi_s},\\
A_s'&=\frac{a(s)\sn_c(\ell_s)^{n-1}\sin\theta}{\cos^3\phi_s}
\bigl((n-1)\cn_c(\ell_s)\cos\phi_s+a(s)\cos\theta\bigr).
\end{aligned}
\]
By the first identity in~\eqref{Eq:ModelIdentities}, the last factor is positive
\[
n\cn_c(\ell_s)\cos\phi_s-\lambda\sn_c(\ell_s)
=(n-1)\cn_c(\ell_s)\cos\phi_s+a(s)\cos\theta>0
\]
where we used that $\cn_c(\ell_s) > 0$.
Consequently, $A_s'>0$ for $\theta  > 0$ and
\begin{equation}
\label{Eq:ModelRr}
\Rr_s(\theta)=
\frac{\sn_c(\ell_s)\cos^2\phi_s}
{n\cn_c(\ell_s)\cos\phi_s-\lambda\sn_c(\ell_s)}.
\end{equation}
To differentiate this quotient, it is convenient to consider its reciprocal:
\[
\frac1{\Rr_s}=
\frac{n\cn_c(\ell_s)}{\sn_c(\ell_s)\cos\phi_s}
-\frac{\lambda}{\cos^2\phi_s}.
\]
Since $\cn_c^2+c\sn_c^2=1$, we have
\[
\left(\frac{\cn_c(\ell_s)}{\sn_c(\ell_s)}\right)'
=-\frac{\ell_s'}{\sn_c(\ell_s)^2}
=-\frac{a(s)\sin\theta}{\sn_c(\ell_s)\cos\phi_s}.
\]
Therefore,
\[
\begin{aligned}
\left(\frac1{\Rr_s}\right)'
={}&-\frac{na(s)\sin\theta}{\sn_c(\ell_s)\cos^2\phi_s}
+\frac{na(s)^2\cn_c(\ell_s)\sin\theta\cos\theta}
{\sn_c(\ell_s)\cos^3\phi_s}-\frac{2\lambda a(s)^2\sin\theta\cos\theta}{\cos^4\phi_s}\\
={}&-\frac{na(s)\sin\theta}{\sn_c(\ell_s)\cos^3\phi_s}
\bigl(\cos\phi_s-a(s)\cn_c(\ell_s)\cos\theta\bigr)-\frac{2\lambda a(s)^2\sin\theta\cos\theta}{\cos^4\phi_s}\\
={}&-\frac{a(s)\sin\theta}{\cos^4\phi_s}
\bigl(nb(s)\cos\phi_s+2\lambda a(s)\cos\theta\bigr),
\end{aligned}
\]
where in the last equality we used the third identity in~\eqref{Eq:ModelIdentities}. Since $\Rr_s'=-\Rr_s^2(1/\Rr_s)'$, we obtain
\[
\begin{aligned}
\Rr_s'(\theta)
={}&\frac{a(s)\sn_c(\ell_s)^2\sin\theta}
{\bigl(n\cn_c(\ell_s)\cos\phi_s-\lambda\sn_c(\ell_s)\bigr)^2} \cdot\bigl(nb(s)\cos\phi_s+2\lambda a(s)\cos\theta\bigr)>0
\end{aligned}
\]
for $0<\theta<\pi/2$, because $a(s)>0$ and $b(s)\ge0$.
\end{proof}

As in the Euclidean proof in Section~\ref{SSec:Touch}, we use the cutoff $d_{\pm e,T}:=\min\{d_{\pm e},T\}$ and set
\begin{equation}
\label{Eq:ModelCutoffArea}
\sA(s,T):=\int_{\S^{n-1}}A_s(d_{\pm e,T}(\omega))\,d\omega.
\end{equation}
Here, $s$ is taken in $[r_0,r]$ in the case of geodesic spheres and horospheres, and $s\in[r_0,\min\{r,\rr_\lambda\}]$ in the equidistant case.

For fixed $s$ with $b(s)>0$, Lemma~\ref{Lem:ModelRr} implies that $\sA(s, \cdot)$ is a continuous, strictly increasing function on $T\in[0,\pi/2]$, with
\[
\sA(s,0)=|\partial B(o,s)|,
\qquad \sA(s,\pi/2)=|\partial L_s|,
\]
where $L_s$ is a $\lambda$-convex lens of inradius $s$. The lenses $L_s$ are nested as $s$ increases, and their surface areas are strictly increasing. 

Since $B(o,r)\subset K$ and $|\partial L_{r_0}|=|\partial K|$, for every such $s$ we have
\[
|\partial B(o,s)|\le|\partial B(o,r)|\le|\partial K|
=|\partial L_{r_0}|\le|\partial L_s|.
\]
If $r\ge\rr_\lambda$ in the equidistant case, we also define the cutoff at $s=\rr_\lambda$. Since $b(\rr_\lambda)=0$, identities~\eqref{Eq:ModelIdentities} give
\[
\sn_c(\ell_{\rr_\lambda}(\theta))
=\frac{\sn_c(\rr_\lambda)}{\cos\theta},
\qquad 0\le\theta<\pi/2.
\]
Moreover, $\cos\phi_{\rr_\lambda}(\theta)\to\lambda/k>0$ as $\theta\to\pi/2$. Thus
\[
\sA(\rr_\lambda,0)=|\partial B(o,\rr_\lambda)|\le|\partial K|,
\qquad
\lim_{T\to\pi/2}\sA(\rr_\lambda,T)=\infty.
\]
The function $T\mapsto\sA(\rr_\lambda,T)$ is continuous and strictly increasing on $[0,\pi/2)$.

Hence, by continuity, strict monotonicity and the intermediate value theorem, there exists a unique cutoff $T=T(s)\in[0,\pi/2]$ for every admissible $s$ (with $T(\rr_\lambda)<\pi/2$ when this endpoint is included) such that
\begin{equation}
\label{Eq:ModelAreaConstraint}
\sA(s,T(s))=|\partial K|.
\end{equation}
For this choice, set
\begin{equation}
\label{Eq:ModelCutoffVolume}
\sM(s):=\int_{\S^{n-1}}V_s(d_{\pm e,T(s)}(\omega))\,d\omega.
\end{equation}
As before, $T(r_0)=\pi/2$ and $\sM(r_0)=|L|$.

\begin{theorem}[Key Comparison]
\label{Thm:ModelKeyA}
In the notation above, 
\[
|K|\ge
\begin{cases}
\sM(\rr_\lambda),&\text{if }c<0,\ 0<\lambda<\sqrt{-c},\text{ and }r\ge\rr_\lambda,\\
\sM(r),&\text{otherwise}.
\end{cases}
\]
\end{theorem}

\begin{proof}
If $b(r)>0$, Lemma~\ref{Lem:ModelRr}, \eqref{Eq:ModelSuperlevels} and \eqref{Eq:ModelAVComparison} allow us to repeat the argument of Section~\ref{SSec:NoTouch}, and to obtain $|K| \ge \sM(r)$. It remains to consider the equidistant case $r\ge\rr_\lambda$. Write $T=T(\rr_\lambda)$. We first show that $V_s-\Rr_{\rr_\lambda}(T)A_s$ is increasing in $s$. This reduces the remaining estimate to a comparison at $s=\rr_\lambda$ where we apply the argument analogous to that in the Euclidean case.

First, $b(\rr_\lambda)=0$ and~\eqref{Eq:ModelRr} give
\[
\Rr_{\rr_\lambda}(\theta)
=\frac{\lambda\cos^2\phi_{\rr_\lambda}(\theta)}
{nk^2\cos^2\phi_{\rr_\lambda}(\theta)-\lambda^2}.
\]
Since $\cos^2\phi_{\rr_\lambda}(\theta)>\lambda^2/k^2$ for $\theta<\pi/2$, we obtain
\begin{equation}
\label{Eq:ModelThresholdSlope}
0<\Rr_{\rr_\lambda}(T)<\frac{\lambda}{(n-1)k^2}.
\end{equation}
If $T=0$, the quotient is understood by continuity.

We next compare the densities at $s=r$ and $s=\rr_\lambda$, keeping $\Theta$ fixed. First, differentiation of~\eqref{Eq:ModelAngle} with respect to $s$, using $a'=-b$, gives
\[
\partial_s(\cos\phi_s)
=\frac{a(s)b(s)\sin^2\theta}{\cos\phi_s}.
\]
Differentiating the first identity in~\eqref{Eq:ModelIdentities} with respect to $s$ gives
\[
-b(s)\partial_s\ell_s
+\cn_c(\ell_s)\partial_s(\cos\phi_s)=-b(s)\cos\theta,
\]
where we used the second identity to collect the coefficient of $\partial_s\ell_s$. Since $b(s)>0$, substitution yields
\begin{equation}
\label{Eq:ModelParameterDerivatives}
\partial_s\ell_s=\cos\theta+
\frac{a(s)\cn_c(\ell_s)\sin^2\theta}{\cos\phi_s},
\qquad
\partial_s(\cos\phi_s)=
\frac{a(s)b(s)\sin^2\theta}{\cos\phi_s}.
\end{equation}

The formulas of~\eqref{Eq:ModelParameterDerivatives} apply whenever $b(s)\ne0$, and extend to $b(s)=0$ by continuity at each fixed admissible angle. For $s\ge\rr_\lambda$, they give
\[
\partial_s\ell_s>0, \qquad \partial_s(\cos\phi_s)\le0, \qquad \partial_s A_s>0.
\]
Here and in the following computation, the argument $\theta$ is suppressed. Using
$b(s)=\lambda\cn_c(\ell_s)-k^2\sn_c(\ell_s)\cos\phi_s$, we find
\[
\begin{aligned}
\partial_s\left(V_s-\frac{\lambda}{(n-1)k^2}A_s\right)
&=-\frac{b(s)\sn_c(\ell_s)^{n-2}}{k^2\cos\phi_s}
\left(\partial_s\ell_s-
\frac{\lambda a(s)\sn_c(\ell_s)\sin^2\theta}{(n-1)\cos^2\phi_s}\right)\\
&=-\frac{b(s)\sn_c(\ell_s)^{n-2}}{k^2\cos^3\phi_s}
\left(\cos\theta+
\frac{n-2}{n-1}\lambda a(s)\sn_c(\ell_s)\sin^2\theta\right)\ge0.
\end{aligned}
\]
For the second equality, substitute $\partial_s\ell_s$ from~\eqref{Eq:ModelParameterDerivatives} and use
\[
\cn_c(\ell_s)\cos\phi_s-\lambda\sn_c(\ell_s)=a(s)\cos\theta,
\qquad
\cos^2\phi_s+a(s)^2\sin^2\theta=1.
\]
Together with~\eqref{Eq:ModelThresholdSlope} and $\partial_sA_s>0$, this proves
\begin{equation}
\label{Eq:ModelThresholdVariation}
\partial_s\bigl(V_s-\Rr_{\rr_\lambda}(T)A_s\bigr)>0
\end{equation}
for every $s\in[\rr_\lambda,r]$ and every angle in the range of $\Theta$.

It remains to make the comparison at $s=\rr_\lambda$. As in the Euclidean proof, let $F_h(t):=|\{\omega \in \S^{n-1} : h(\omega)>t\}|$. Fubini's theorem and the area constraint~\eqref{Eq:ModelAreaConstraint} yield
\[
\begin{aligned}
&\int_{\S^{n-1}}\bigl(V_{\rr_\lambda}(\Theta)
-\Rr_{\rr_\lambda}(T)A_{\rr_\lambda}(\Theta)\bigr)\,d\omega
-\sM(\rr_\lambda)+\Rr_{\rr_\lambda}(T)|\partial K|\\
&\quad=\int_0^{\pi/2} A_{\rr_\lambda}'(t)
\bigl(\Rr_{\rr_\lambda}(t)-\Rr_{\rr_\lambda}(T)\bigr)
\bigl(F_\Theta(t)-F_{d_{\pm e,T}}(t)\bigr)\,dt\ge0.
\end{aligned}
\]
Indeed, for $t<T$, both factors in parentheses are nonpositive, by Lemma~\ref{Lem:ModelRr} and~\eqref{Eq:ModelSuperlevels}. For $t>T$, both are nonnegative. 

Finally, combine~\eqref{Eq:ModelAVComparison}, the integrated form of~\eqref{Eq:ModelThresholdVariation}, and the last inequality:
\[
\begin{aligned}
|K|-\Rr_{\rr_\lambda}(T)|\partial K|
&\ge\int_{\S^{n-1}}\bigl(V_r(\Theta)
-\Rr_{\rr_\lambda}(T)A_r(\Theta)\bigr)\,d\omega\\
&\ge\int_{\S^{n-1}}\bigl(V_{\rr_\lambda}(\Theta)
-\Rr_{\rr_\lambda}(T)A_{\rr_\lambda}(\Theta)\bigr)\,d\omega\\
&\ge\sM(\rr_\lambda)-\Rr_{\rr_\lambda}(T)|\partial K|.
\end{aligned}
\]
This gives $|K|\ge\sM(\rr_\lambda)$, with strict inequality if $r>\rr_\lambda$.
\end{proof}

To compare with $L$, we show that $\sM(s)$ is strictly increasing as in the Euclidean case in Lemma~\ref{Lem:M}.

\begin{lemma}
\label{Lem:ModelM}
The function $s\mapsto\sM(s)$ is strictly increasing on its domain of definition.
\end{lemma}
\begin{proof}
Write $T=T(s)$. By rotational symmetry,
\[
\begin{aligned}
\frac{\sA(s,T)}{2|\S^{n-2}|}
&=\int_0^T A_s(\theta)\sin^{n-2}\theta\,d\theta
+A_s(T)\int_T^{\pi/2}\sin^{n-2}\theta\,d\theta,\\
\frac{\sM(s)}{2|\S^{n-2}|}
&=\int_0^T V_s(\theta)\sin^{n-2}\theta\,d\theta
+V_s(T)\int_T^{\pi/2}\sin^{n-2}\theta\,d\theta.
\end{aligned}
\]
Here $0<T<\pi/2$, and
\[
\partial_T\sA(s,T)=2|\S^{n-2}|A_s'(T)
\int_T^{\pi/2}\sin^{n-2}\theta\,d\theta>0.
\]
Thus $T(s)$ is differentiable by the implicit function theorem. Differentiating the area constraint and $\sM$, then eliminating $T'(s)$ exactly as in Lemma~\ref{Lem:M}, gives
\begin{equation}
\label{Eq:ModelVariation}
\begin{aligned}
\frac{\sM'(s)}{2|\S^{n-2}|}
={}&\int_0^T\bigl(\partial_sV_s(\theta)
-\Rr_s(T)\partial_sA_s(\theta)\bigr)\sin^{n-2}\theta\,d\theta\\
&+\bigl(\partial_sV_s(T)-\Rr_s(T)\partial_sA_s(T)\bigr)
\int_T^{\pi/2}\sin^{n-2}\theta\,d\theta.
\end{aligned}
\end{equation}

We now calculate the two terms in~\eqref{Eq:ModelVariation}. As in the proof of Lemma~\ref{Lem:ModelRr}, the argument $\theta$ in $\ell_s$ and $\phi_s$ is suppressed when no other argument is displayed. Using \eqref{Eq:ModelParameterDerivatives}, we can write the integrands in~\eqref{Eq:ModelVariation} as total derivatives in $\theta$:
\begin{equation}
\label{Eq:ModelTotalDerivatives}
\begin{aligned}
\bigl(\partial_sV_s(\theta)\bigr)\sin^{n-2}\theta
&=\frac1{n-1}\frac{d}{d\theta}
\bigl(\sn_c(\ell_s)^{n-1}\sin^{n-1}\theta\bigr),\\
\bigl(\partial_sA_s(\theta)\bigr)\sin^{n-2}\theta
&=\frac{d}{d\theta}
\left(\frac{\cn_c(\ell_s)\sn_c(\ell_s)^{n-2}}
{\cos\phi_s}\sin^{n-1}\theta\right).
\end{aligned}
\end{equation}
Let us verify both identities. For the volume identity, use
$\ell_s'=a(s)\sn_c(\ell_s)\sin\theta/\cos\phi_s$ to obtain
\[
\begin{aligned}
&\frac1{n-1}\frac{d}{d\theta}
\bigl(\sn_c(\ell_s)^{n-1}\sin^{n-1}\theta\bigr)\\
&\qquad=\sn_c(\ell_s)^{n-1}\sin^{n-2}\theta
\left(\cos\theta+
\frac{a(s)\cn_c(\ell_s)\sin^2\theta}{\cos\phi_s}\right)\\
&\qquad=\sn_c(\ell_s)^{n-1}\partial_s\ell_s\sin^{n-2}\theta
=\partial_sV_s\sin^{n-2}\theta.
\end{aligned}
\]
For the area identity, differentiation of $A_s=\sn_c(\ell_s)^{n-1}/\cos\phi_s$ with respect to $s$ gives
\[
\begin{aligned}
\partial_sA_s
={}&(n-1)\frac{\cn_c(\ell_s)\sn_c(\ell_s)^{n-2}}
{\cos\phi_s}\cos\theta\\
&+\frac{a(s)\sn_c(\ell_s)^{n-2}\sin^2\theta}{\cos^2\phi_s}
\left((n-1)\cn_c(\ell_s)^2
-\frac{b(s)\sn_c(\ell_s)}{\cos\phi_s}\right).
\end{aligned}
\]
On the other hand, the product rule gives, for $\theta>0$,
\[
\begin{aligned}
&\frac1{\sin^{n-2}\theta}\frac{d}{d\theta}
\left(\frac{\cn_c(\ell_s)\sn_c(\ell_s)^{n-2}}
{\cos\phi_s}\sin^{n-1}\theta\right) =(n-1)\frac{\cn_c(\ell_s)\sn_c(\ell_s)^{n-2}}
{\cos\phi_s}\cos\theta\\
&\qquad+\frac{a(s)\sn_c(\ell_s)^{n-2}\sin^2\theta}{\cos^2\phi_s}
\left((n-2)\cn_c(\ell_s)^2-c\sn_c(\ell_s)^2
+\frac{a(s)\cn_c(\ell_s)\cos\theta}{\cos\phi_s}\right).
\end{aligned}
\]
The expressions in parentheses in the last two equations agree. Indeed, by $\cn_c^2+c\sn_c^2=1$ and~\eqref{Eq:ModelIdentities},
\[
\begin{aligned}
&(n-2)\cn_c(\ell_s)^2-c\sn_c(\ell_s)^2+\frac{a(s)\cn_c(\ell_s)\cos\theta}{\cos\phi_s}\\
&\qquad=(n-1)\cn_c(\ell_s)^2-1
+\frac{\cos\phi_s-b(s)\sn_c(\ell_s)}{\cos\phi_s}\\
&\qquad=(n-1)\cn_c(\ell_s)^2
-\frac{b(s)\sn_c(\ell_s)}{\cos\phi_s}.
\end{aligned}
\]
This proves the area identity in~\eqref{Eq:ModelTotalDerivatives}; the identity at $\theta=0$ follows by continuity.

Integrating the two total derivatives in~\eqref{Eq:ModelVariation}, we obtain
\[
\begin{aligned}
\frac{\sM'(s)}{2|\S^{n-2}|}
={}&\left(\frac{\sn_c(\ell_s(T))^{n-1}}{n-1}
-\Rr_s(T)\frac{\cn_c(\ell_s(T))\sn_c(\ell_s(T))^{n-2}}
{\cos\phi_s(T)}\right)\sin^{n-1}T\\
&+\bigl(\partial_sV_s(T)-\Rr_s(T)\partial_sA_s(T)\bigr)
\int_T^{\pi/2}\sin^{n-2}\theta\,d\theta.
\end{aligned}
\]
It remains to simplify the two coefficients. At a general angle $\theta$, formula~\eqref{Eq:ModelRr} gives
\[
\begin{aligned}
&\frac{\sn_c(\ell_s)^{n-1}}{n-1}
-\Rr_s(\theta)\frac{\cn_c(\ell_s)\sn_c(\ell_s)^{n-2}}{\cos\phi_s} =\frac{\sn_c(\ell_s)^{n-1}
\bigl(\cn_c(\ell_s)\cos\phi_s-\lambda\sn_c(\ell_s)\bigr)}
{(n-1)\bigl(n\cn_c(\ell_s)\cos\phi_s-\lambda\sn_c(\ell_s)\bigr)}\\
&\quad=\frac{a(s)\sn_c(\ell_s)^{n-1}\cos\theta}
{(n-1)\bigl(n\cn_c(\ell_s)\cos\phi_s-\lambda\sn_c(\ell_s)\bigr)}.
\end{aligned}
\]
For the second coefficient, differentiate the area density before substituting the derivatives of $\ell_s$ and $\cos\phi_s$. Then~\eqref{Eq:ModelRr} yields
\[
\begin{aligned}
&\partial_sV_s-\Rr_s(\theta)\partial_sA_s\\
&\quad=\frac{\sn_c(\ell_s)^{n-1}}
{n\cn_c(\ell_s)\cos\phi_s-\lambda\sn_c(\ell_s)} \cdot\left(
\bigl(\cn_c(\ell_s)\cos\phi_s-\lambda\sn_c(\ell_s)\bigr)
\partial_s\ell_s+\sn_c(\ell_s)\partial_s(\cos\phi_s)\right)\\
&\quad=\frac{a(s)\sn_c(\ell_s)^{n-1}}
{n\cn_c(\ell_s)\cos\phi_s-\lambda\sn_c(\ell_s)}
\left(\cos\theta\,\partial_s\ell_s
+\frac{b(s)\sn_c(\ell_s)\sin^2\theta}{\cos\phi_s}\right).
\end{aligned}
\]
The equation in parentheses is equal to one. Indeed, by~\eqref{Eq:ModelParameterDerivatives} and~\eqref{Eq:ModelIdentities},
\[
\begin{aligned}
\cos\theta\,\partial_s\ell_s
+\frac{b(s)\sn_c(\ell_s)\sin^2\theta}{\cos\phi_s}
&=\cos^2\theta+
\frac{\sin^2\theta}{\cos\phi_s}
\bigl(a(s)\cn_c(\ell_s)\cos\theta+b(s)\sn_c(\ell_s)\bigr)\\
&=\cos^2\theta+\sin^2\theta=1.
\end{aligned}
\]
Evaluating these two coefficients at $\theta=T$ therefore gives
\begin{equation}
\label{Eq:ModelPositiveDerivative}
\begin{aligned}
\frac{\sM'(s)}{2|\S^{n-2}|}
={}&\frac{a(s)\sn_c(\ell_s(T))^{n-1}}
{n\cn_c(\ell_s(T))\cos\phi_s(T)-\lambda\sn_c(\ell_s(T))}\\
&\quad\cdot\left(\int_T^{\pi/2}\sin^{n-2}\theta\,d\theta
+\frac{\cos T\sin^{n-1}T}{n-1}\right)>0.
\end{aligned}
\end{equation}
When the endpoint $\rr_\lambda$ is included in the domain of $\sM$, we have $\sM(s)\to\sM(\rr_\lambda)$ as $s \to \rr_\lambda$. This completes the proof. 
\end{proof}

If $b(r)>0$, Theorem~\ref{Thm:ModelKeyA} and Lemma~\ref{Lem:ModelM} give
\[
|K|\ge\sM(r)\ge\sM(r_0)=|L|.
\]
If $r>r_0$, the second inequality is strict. If $r=r_0$, the equality case in Theorem~\ref{thm:reverse-inradius} implies that $K$ is a $\lambda$-convex lens.
In the equidistant case with $r\ge\rr_\lambda$, the same results yield
\[
|K|\ge\sM(\rr_\lambda)>\sM(r_0)=|L|,
\]
since $r_0<\rr_\lambda$. This proves the Main Theorem in spherical and hyperbolic spaces for $\lambda$-convex polytopes. The approximation argument completes the proof for arbitrary $\lambda$-convex bodies.

\bigskip

\noindent
{\bf Acknowledgments.} 
We thank Julian Scheuer and Deniz Hamdy for discussions regarding the reverse isoperimetric problem. Julian informed us of Deniz’s proposed solution to the problem in Euclidean space, obtained with the assistance of LLM. Our proof was developed independently of Hamdy’s argument, which we have not verified.

\bigskip

\noindent
{\bf AI assistance.} ChatGPT was used in exploratory discussions and in the preparation and editing of this manuscript.

\end{document}